\documentclass[12pt]{amsart}

\usepackage{enumerate, amsmath, amsthm, amsfonts, amssymb, xy,  mathrsfs, graphicx, paralist, fancyvrb, ytableau, color}
\usepackage[margin=1in]{geometry} 
\usepackage[bookmarks, colorlinks=true, linkcolor=blue, citecolor=blue, urlcolor=blue]{hyperref}

\input xy
\xyoption{all}

\numberwithin{equation}{subsection}
\newtheorem{theorem}[equation]{Theorem}

\newtheorem{proposition}[equation]{Proposition}
\newtheorem{lemma}[equation]{Lemma}
\newtheorem{corollary}[equation]{Corollary}
\newtheorem{conjecture}[equation]{Conjecture}

\theoremstyle{definition}
\newtheorem{rmk}[equation]{Remark}

\newtheorem{eg}[equation]{Example}

\newtheorem{defn}[equation]{Definition}
\newenvironment{definition}[1][]{\begin{defn}[#1]\pushQED{\qed}}{\popQED \end{defn}}

\newenvironment{subeqns}[1][]{\addtocounter{equation}{-1}
\begin{subequations}

}{\end{subequations}}

\newcommand{\cA}{\mathcal{A}}

\newcommand{\cB}{\mathcal{B}}

\newcommand{\cO}{\mathcal{O}}

\newcommand{\bS}{\mathbf{S}}

\renewcommand{\phi}{\varphi}

\newcommand{\arxiv}[1]{\href{http://arxiv.org/abs/#1}{{\tt arXiv:#1}}}

\DeclareMathOperator{\rank}{rank}

\DeclareMathOperator{\Sym}{Sym}

\DeclareMathOperator{\quot}{\mathsf{Quot}}

\begin{document}
 \baselineskip=17pt
\parindent=30pt

\title[On the cohomology of tautological bundles]{On the cohomology of tautological bundles\\ over Quot schemes of curves}

\author{Alina Marian}
\address{Department of Mathematics, Northeastern University}
\email{a.marian@northeastern.edu}
\author{Dragos Oprea}
\address{Department of Mathematics, University of California, San Diego}
\email {doprea@math.ucsd.edu}
\author{Steven V Sam}
\address{Department of Mathematics, University of California, San Diego}
\email {ssam@ucsd.edu}

\begin{abstract}
We consider tautological bundles and their exterior and symmetric powers on the Quot scheme over the projective line. We prove and conjecture several statements regarding the vanishing of their higher cohomology, and we describe their spaces of global sections via tautological constructions. To this end, we make use of the embedding of the Quot scheme as an explicit local complete intersection in the product of two Grassmannians, studied by Str{\o}mme. This allows us to construct resolutions with vanishing cohomology for the tautological bundles and their exterior and symmetric powers. We further illustrate our approach with a few additional cohomological calculations. \end{abstract}

\maketitle

\section{Introduction}

\subsection{Tautological vector bundles}

We consider the Quot scheme $\quot_{\, \mathbb P^1} (\mathbb C^N, n, r)$ \\ parametrizing rank $r$ degree $n$ quotients of the trivial bundle of rank $N$ over $\mathbb P^1$: $$0\to S\to \mathbb C^N\otimes \mathcal O_{\mathbb P^1} \to Q\to 0,\quad \text{rank }Q=r, \quad \text{deg }Q=n.$$ It is a smooth projective variety which comes equipped with the universal sequence $$0\to \mathcal {S}  \to \mathbb C^N\otimes \mathcal O \to \mathcal Q\to 0$$ over $\mathbb P^1\times \quot_{\, \mathbb P^1} (\mathbb C^N, n, r)$. We let $p$ and $\pi$ denote the projections to the two factors.

The Quot scheme over the projective line is an important testing ground for ideas in moduli theory. It has rich geometry and bears ties to homogeneous and quiver varieties while not being one of them. A beautiful systematic study of $\quot_{\, \mathbb P^1} (\mathbb C^N, n, r)$ was carried out in \cite{St}, where the Quot scheme was shown to be rational, and the Betti numbers, generators for the Chow ring, and the nef cone were calculated. As noted in \cite{St}, $\quot_{\, \mathbb P^1}  (\mathbb C^N, n, r)$ is a compactification of the space of degree $d$ morphisms from $\mathbb P^1$ to the Grassmannian ${\mathbf G} (N, r)$ of $r$ dimensional quotients of $\mathbb C^N$. With this point of view, in the 1990s, the Quot scheme was used effectively to calculate the small quantum cohomology ring of the Grassmannian, leading eventually to a calculation for all flag varieties \cite{bertram, cf1, cf2, K, chen}. Further progress included the description of the equivariant cohomology ring in \cite{bcs}, the calculation of the effective cone \cite {J}, and the birational study \cite{ito} which established $\quot_{\, \mathbb P^1} (\mathbb C^N, n, r)$ as a Mori dream space, among others.

In this paper, we take up the problem of calculating the cohomology of Schur functors of tautological bundles over $\quot_{\, \mathbb P^1} (\mathbb C^N, n, r)$. While our results are primarily in the setting of zero quotient rank ($r=0$), the method is available for any rank. Conjectures for any $r$ are formulated in Section \ref{xxyy} below.

To start, note that for any line bundle $L\to \mathbb P^1$, there is an induced tautological complex of rank $n + r (\deg L + 1)$ over $\quot_{\, \mathbb P^1} (\mathbb C^N, r, n)$ given by
\begin{equation}\label{taut}
  L^{[n]}= R\pi_\star(p^\star L\otimes \mathcal Q).
\end{equation}
When $r=0,$ $L^{[n]}$ is a {\it vector bundle} of rank $n$. 
(Note that $R^1\pi_\star(p^\star L \otimes \mathcal Q) = 0$ for $r=0$ since the support of $\mathcal Q$ is finite in each fiber of $\pi$.)
We let $\quot_{\,\mathbb P^1}(\mathbb C^N, n)$ denote the Quot scheme in this case, and show the following results. 

\begin{theorem}\label{ta}
\textnormal{(1)} For all line bundles $L\to \mathbb P^1$ with $\deg L\geq n\geq k$, we have $$H^0\bigg(\quot_{\,\mathbb P^1}(\mathbb C^N, n), \bigwedge^k L^{[n]}\bigg)\cong \bigwedge^k H^0(L^{\oplus N})$$ and the higher cohomology vanishes $$H^i\bigg(\quot_{\,\mathbb P^1}(\mathbb C^N, n), \bigwedge^k L^{[n]}\bigg)=0, \quad i>0.$$ 
\textnormal{(2)} More generally, assume $\deg L\geq n\geq k$ and let $p_1, \ldots, p_t$ be nonnegative integers, $0\leq t\leq N-1$. We have 
$$\textnormal{Ext}^{i} \bigg(\bigwedge^{p_1} L^{[n]}\otimes \cdots \otimes \bigwedge^{p_t} L^{[n]}, \bigwedge^{k} L^{[n]}\bigg)=\begin{cases} \bigwedge^{k-|p|} H^0(L^{\oplus N}) & \text{ if } i=0\\0 &\text{ if } i>0\end{cases},$$ for  $|p|=p_1+\cdots+p_t\leq k$. If $|p|>k$, all the above Ext groups vanish. 
\end{theorem}

\begin{theorem}\label{tb}
\textnormal{(1)} For all line bundles $L\to \mathbb P^1$ with $\deg L\geq n\geq k$, we have $$H^0\left(\quot_{\,\mathbb P^1}(\mathbb C^N, n), \Sym^k L^{[n]}\right)\cong \Sym^{k} H^0(L^{\oplus N})$$ and the higher cohomology vanishes $$H^i\left(\quot_{\,\mathbb P^1}(\mathbb C^N, n), \Sym^k L^{[n]}\right)=0, \quad i>0.$$
\textnormal{(2)} More generally, assume $\deg L\geq n\geq k$ and let $p_1, \ldots, p_t$ be nonnegative integers, $0\leq t\leq N-1$. Then 
$$\textnormal {Ext}^i \bigg(\bigwedge^{p_1} L^{[n]}\otimes \cdots \otimes \bigwedge^{p_t} L^{[n]}, \textnormal{Sym}^k \,L^{[n]}\bigg)=\begin{cases}  \textnormal{Sym}^{k-|p|}\, H^0(L^{\oplus N}) &\text{ if } i=0 \text{ and all } p_j\in \{0, 1\}\\ 0 &\text{ otherwise}\end{cases},$$ for $|p|\leq k$. If $|p|>k$, all the above Ext groups vanish. 
\end{theorem}

We expect that the vanishing of higher cohomology of $\bigwedge^k L^{[n]}$ and $\text{Sym}^kL^{[n]}$ in Theorems \ref{ta} and \ref{tb} holds whenever $\deg L\geq -1$. For arbitrary exterior powers, the above vanishings cannot be accessed by the classical theorems. In the special case of the determinant line bundle $\bigwedge^n L^{[n]}$, the bound in our theorems improves the bound $\deg L\geq Nn-N-n$ obtained by Kodaira vanishing. The latter can be applied using the description of the ample cone in \cite {St} and a standard calculation of the canonical bundle via Grothendieck--Riemann--Roch.
\vskip.1in
It is easy to see how the sections of the bundles $\bigwedge^k L^{[n]}$ and $\Sym^k L^{[n]}$ over $\mathsf{Quot}_{\mathbb P^1}(\mathbb C^N, n)$ arise geometrically. Indeed, from the universal quotient $$\mathbb C^N\otimes \mathcal O\to \mathcal Q\to 0$$ over $\mathbb P^1\times \quot_{\,\mathbb P^1}(\mathbb C^N, n)$, after tensoring by $L$ and pushing forward, we immediately obtain a map $$H^0(L^{\oplus N})\otimes \mathcal O_{\mathsf{Quot}}\to L^{[n]}.$$ Taking exterior and symmetric powers, and taking cohomology, we obtain morphisms \begin{equation}\label{phipsi}
  \Phi_k \colon \bigwedge^k H^0(L^{\oplus N})\to H^0\bigg(\bigwedge^k L^{[n]}\bigg), \quad
  \Psi_k \colon \Sym^k H^0(L^{\oplus N})\to H^0\left(\Sym^k L^{[n]}\right).
\end{equation}
Our proofs will show that $\Phi_k$ and $\Psi_k$ are isomorphisms when $\deg L\geq n\geq k$. Thus all sections of $\bigwedge^k L^{[n]}$ and $\Sym^k L^{[n]}$ are obtained tautologically, while the higher cohomology vanishes. 
\vskip.1in
We further illustrate the techniques developed here by showing that: 

\begin{theorem}\label{tc}
  For all line bundles $L, M\to \mathbb P^1$ with $\deg M\geq n$ and $0\leq \deg M-\deg L\leq 1$, for all $p_1, \ldots, p_t\geq 0$, not all zero, and $1\leq t\leq N-1$, we have
  \[
    H^i\bigg(\quot_{\,\mathbb P^1}(\mathbb C^N, n), \big(\bigwedge^{p_1} L^{[n]}\big)^{\vee}\otimes (\bigwedge^{p_2} M^{[n]}\big)^{\vee}\otimes \cdots \otimes \big(\bigwedge^{p_t} M^{[n]}\big)^{\vee}\bigg)=0, \quad i\geq 0.
    \]
\end{theorem} 
\noindent Of course, the case $L=M$ is contained in Theorems \ref{ta}\,(2) and \ref{tb}\,(2) for $k=0$. \vskip.1in

Although we focus for notational simplicity on the case of quotients of the trivial bundle, the same arguments easily apply to the Quot scheme $\quot_{\,\mathbb P^1}(E, n)$ of finite quotients of an arbitrary vector bundle $E\to \mathbb P^1$. 

Let $a$ denote the largest degree appearing in the splitting of $E$ as a direct sum of line bundles, and set $\alpha=-\deg E+(N-1)a$. For Theorems \ref{ta} and \ref{tb}, we replace all instances of $H^0(L^{\oplus N})$ by $H^0(E\otimes L)$. For instance, Theorem \ref{ta}\,(1) takes the form 
 $$H^0\bigg(\quot_{\,\mathbb P^1}(E, n), \bigwedge^k L^{[n]}\bigg)\cong \bigwedge^k H^0(E\otimes L), \quad H^i\bigg(\quot_{\,\mathbb P^1}(E, n), \bigwedge^k L^{[n]}\bigg)=0, \quad i>0$$ whenever $$\deg L\geq n +\alpha, \quad n\geq k\geq 0.$$ Remark \ref{rembd} below explains how this bound emerges. Similarly, the analogue of Theorem \ref{tb}\,(1) reads $$H^0\left(\quot_{\,\mathbb P^1}(E, n), \Sym^k L^{[n]}\right)\cong \Sym^{k} H^0(E\otimes L), \quad H^i\left(\quot_{\,\mathbb P^1}(E, n), \Sym^k L^{[n]}\right)=0, \quad i>0.$$ The more general Theorem \ref{ta}\,(2) and Theorem \ref{tb}\,(2) are also correct after the analogous modifications. Likewise, Theorem \ref{tc} remains true under the assumption $\deg M\geq n+\alpha$.
 
\vskip.1in

We note that for a smooth projective curve $C$ of arbitrary genus, the holomorphic Euler characteristics of $\bigwedge^k L^{[n]}$, $\text{Sym}^k L^{[n]}$ and $\bigg(\bigwedge^p L^{[n]}\bigg)^{\vee}$ on $\quot_C (\mathbb C^N, n)$ were  calculated in \cite{OS} by equivariant localization. The following expectation regarding individual cohomology groups was also formulated in \cite [Question 20]{OS}:
\begin{equation}\label{spec} H^{\bullet}\bigg (\quot_{\,C }(\mathbb C^N, n), \bigwedge^{k} L^{[n]}\bigg)\cong \bigwedge^{k} H^{\bullet}(L^{\oplus N})\otimes \Sym^{n-k} H^{\bullet} (\mathcal O_C).\end{equation} The exterior and symmetric powers on the right hand side are understood in the graded sense. 

In the case when $C = \mathbb P^1,$ our theorems confirm this expectation for $\deg L \geq n$. We offer additional modest evidence for $k=1$ and $L$ of arbitrary degree:
\begin{corollary}\label{t2}
For all line bundles $L\to \mathbb P^1$, we have 
$$H^i\left(\quot_{\,\mathbb P^1}(\mathbb C^N, n), L^{[n]}\right)=0, \quad i\geq 2.$$
\end{corollary}
 
The above results and formula \eqref{spec} reflect a full parallelism to the cohomology of tautological bundles over the Hilbert scheme of points on a surface computed in \cite{D, Sc, Sc2, Kr, A}. For instance, for all line bundles $L\to X$ over smooth projective surfaces, we have 
\[
  H^{\bullet}\big(X^{[n]}, \bigwedge^k L^{[n]}\big) = \bigwedge^k H^{\bullet}(X, L) \otimes \Sym^{n-k} H^{\bullet}(X, \mathcal O_X).
\]
The Bridgeland--King--Reid correspondence plays a central role in the proof. We refer the reader to the beautiful article \cite{Kr} for state-of-the-art calculations in the surface case. \vskip.1in

\subsection{Proofs} To establish Theorems \ref{ta}, \ref{tb}, \ref{tc}, the key idea is to use the twofold Grothendieck embedding of the Quot scheme into a product of Grassmannians,
\[
  \iota \colon \mathsf{Quot}_{\mathbb P^1} (\mathbb C^N, n)\hookrightarrow \mathbf G_1\times \mathbf G_2,
\]
so that the image of $\iota$ is the zero locus of a regular section $\sigma$ of an explicit homogeneous vector bundle $$\mathcal E\to \mathbf G_1\times \mathbf G_2.$$ This embedding as an explicit local complete intersection is specific to the Quot scheme (of quotients of all ranks) over the projective line. It was considered and studied in detail by 
 Str{\o}mme \cite {St} who used it to derive information about the Chow ring. Our paper widens the picture, and shows that the embedding is also well-suited to the study of the tautological bundles. 

Using the Koszul resolution for $\sigma$ $$\cdots \to \bigwedge^{2}\mathcal E^{\vee}\to \bigwedge^1\mathcal E^\vee\to \mathcal O_{\bf {G_1}\times {\bf G_2}}\to \mathcal O_{\text{Quot}}\to 0\,,$$ we obtain resolutions 
\[
  \cdots\to \mathcal R_2\to \mathcal R_1\to \mathcal R_0\to \mathcal \iota_{\star} \mathcal F\to 0
\]
for each one of the tautological bundles $\mathcal F$ appearing in Theorems \ref{ta}, \ref{tb} and \ref{tc}. The resolutions thus obtained are special. Remarkably, we show that the terms $\mathcal R_\ell$ have vanishing cohomology for all $\ell\geq 1$, while $\mathcal R_0$ has no higher cohomology. This allows us to control the cohomology of the tautological bundles and establish our results. 

The argument makes crucial use of the Borel--Weil--Bott theorem on the two Grassmannians ${\bf G}_1, {\bf G}_2$, along with several combinatorial arguments involving the Littlewood--Richardson rule. In intermediate stages, statements of independent interest are established generally over arbitrary Grassmannians; we refer the reader to Section \ref{s3} for details. It takes a delicate interplay between Borel--Weil--Bott and Littlewood--Richardson vanishings to show that all higher terms $\mathcal R_\ell$, $\ell\geq 1$ of the resolution have no cohomology at all. 

\vskip.1in
While the above theorems concern genus $0$, we also obtain the following corollary in arbitrary genus. Let $y$ be a variable. Setting
$$\bigwedge_y V:=\sum_k y^k \bigwedge^kV, \quad \Sym_y V:=\sum_k y^k \,\text{Sym}^{k}V,$$
the result below recovers Theorem $1$, a special case of Theorem $2$, and Theorem $4$ in \cite {OS}.
\begin{corollary}\label{tos}
  \begin{subeqns}
		Let $L, M_1, \ldots, M_t\to C$ be line bundles over a smooth projective curve, where $1\leq t\leq N-1$. Then 
                \begin{align}
                  \label{wedge}\sum_{n=0}^{\infty} &q^n\chi\bigg(\quot_{\,C }(\mathbb C^N, n), \bigwedge_yL^{[n]}\bigg)={(1-q)^{-\chi(\cO_C)}}{(1+qy)^{N\chi(L)}},\\
		\label{dual}\sum_{n=0}^{\infty} &q^n\chi\bigg(\quot_{\,C }(\mathbb C^N, n), \bigotimes_{i=1}^{t}\big(\bigwedge_{y_i} M_i^{[n]} \big)^{\vee}\bigg)=(1-q)^{-\chi(\cO_C)}.
		\end{align}
	Furthermore, in genus $0$, we have 	
\begin{align}
\label{sym}\sum_{n\geq k} &q^ny^k \chi\left(\quot_{\,\mathbb P^1}(\mathbb C^N, n), \Sym^k L^{[n]}\right)=(1-q)^{-1}(1-qy)^{-N\chi(L)}.
        \end{align}
      \end{subeqns}
    \end{corollary}	

We will show how to derive Corollary \ref{tos} from Theorems \ref{ta}, \ref{tb} and \ref{tc} in Section \ref{cortos}. 

Formulas \eqref{wedge}, \eqref{dual} and \eqref{sym} were previously established in \cite{OS} based on reduction to genus $0$ using universality statements as in \cite{EGL, OS, stark}, and equivariant torus localization in genus $0$. The localization calculation is combinatorially involved and relies on several mysterious simplifications. In the present paper, Theorems \ref{ta}, \ref{tb} and \ref{tc} reflect an efficient and  more conceptual approach to the full cohomology, which cannot be accessed by localization.

 In all genera, Corollary \ref{tos} also holds for an arbitrary vector bundle $E$ instead of the trivial bundle $\mathbb C^N\otimes \mathcal O_{\mathbb P^1}$, with the only modification that all instances of $N\chi(L)$ are replaced by $\chi(E\otimes L)$.

\subsection{Higher rank}\label{xxyy} A natural direction is to apply the techniques of this paper to study the cohomology of the Schur functors of tautological bundles over $\quot_{\, \mathbb P^1} (\mathbb C^N, n, r)$, when $r > 0.$ In this setting, even the numerical $K$-theoretic invariants of these Schur functors are largely unexplored. 

Recalling the definition of the tautological complex $L^{[n]}$ from \eqref{taut}, we propose the following conjectures.

\begin{conjecture}\label{cw}
  Let $n=(N-r)a+b$ with $0\leq b<N-r$. Then for all line bundles $L\to \mathbb P^1$, we have $$\chi\bigg(\quot_{\,\mathbb P^1}(\mathbb C^N, n, r), \bigwedge^k L^{[n]}\bigg)=\binom{N\chi(L)}{k}$$ for all $k\leq n+r(a+1).$ 
\end{conjecture}

\noindent For symmetric powers, we state the following

\begin{conjecture} \label{cs}
  Let $n=(N-r)a+b$ with $0\leq b<N-r$. Then for all line bundles $L\to \mathbb P^1$, we have $$\chi\bigg(\quot_{\,\mathbb P^1}(\mathbb C^N, n, r), \Sym^k L^{[n]}\bigg)=\binom{N\chi(L)+k-1}{k}$$ for all $k\leq n+r(a+1).$
\end{conjecture}

\noindent Finally, for the dualized exterior powers, we have 

\begin{conjecture}
  Let $r>0$ and write $n=ar+b$ with $0\leq b<r$. Let $1\leq t\leq N-r-1$ and $p_1, \ldots, p_t$ be nonnegative integers with $0<p_1+\cdots+p_t\leq n+(N-r)(a+1)$. Then for all line bundles $L_1, \ldots, L_t\to \mathbb P^1$, we have
  \[
    \chi\bigg(\quot_{\,\mathbb P^1}(\mathbb C^N, n, r), \big(\bigwedge^{p_1} L_1^{[n]}\big)^{\vee}\otimes \cdots \otimes \big(\bigwedge^{p_t} L_t^{[n]}\big)^{\vee}\bigg)=0.
  \]
\end{conjecture}

When $r=0$, Conjectures \ref{cw} and \ref{cs} recover Theorems \ref{ta}\,(1) and \ref{tb}\,(1), while the case $r=N-1$ can be verified by hand since the Quot scheme is a projective space. For all three conjectures, we checked the answer by computer in several other cases. The bounds on $k$ appear to be sharp.

It is natural to expect that the three conjectures can be lifted in obvious fashion to cohomology. This suggests that Theorem \ref{ta}\,(1), \ref{tb}\,(1) and \ref{tc} continue to hold for higher rank quotients subject to the appropriate bounds on $k$ and the condition that the degree of the $L$'s be nonnegative. While Str{\o}mme's construction is valid for quotients of arbitrary rank, the Borel--Weil--Bott arguments become more involved and will be left for future study.

The answers predicted by the conjectures stabilize as $n$ becomes large with respect to $N, k, r$. Equivalently, for each fixed $k$, the generating series $$\sum_{n=0}^{\infty} q^n \chi\bigg(\quot_{\,\mathbb P^1}(\mathbb C^N, n, r), \bigwedge^k L^{[n]}\bigg), \quad \sum_{n=0}^{\infty} q^n \chi\left(\quot_{\,\mathbb P^1}(\mathbb C^N, n, r), \Sym^k L^{[n]}\right)$$ are given by rational functions with (simple) pole only at $q=1$. It is natural to wonder whether this statement is correct for all partitions and for the associated Schur functors of $L^{[n]}$. 

\subsection {Subsequent developments} After our preprint appeared on \texttt{arxiv}, the conjectural identity \eqref{spec} was extended to cover more general Ext groups involving tensor products of wedge powers of tautological bundles, see \cite [Conjecture 1.1]{Krug}. The extended conjecture is consistent with the Euler characteristic calculations in \cite {OS}. Theorem \ref{ta} and Theorem \ref{tc} establish part of the conjecture. We also refer the reader to \cite [Theorems 1.2, 1.3, 1.5]{Krug} for related results. Both equation \eqref{spec} and the conjecture formulated in \cite {Krug} were very recently confirmed in \cite {MN} using different methods.

\subsection{Plan of the paper.} We review Str{\o}mme's embedding, construct the resolutions of the tautological bundles, and establish Theorems \ref{ta}, \ref{tb} and \ref{tc} in Section \ref{s2}. This relies on the Borel--Weil--Bott analysis of the resolutions which is carried out in Section \ref{s3}. Corollaries \ref{tos} and \ref{t2} are proved in Section \ref{cor}.

\subsection{Acknowledgements} We thank Shubham Sinha and Jerzy Weyman for useful related conversations. A.M. was supported by NSF grant DMS 1902310, D.O. was supported by NSF grant DMS 1802228, and S.S. was supported by NSF grant DMS 1812462. 

\section {Grassmannian embedding of the Quot scheme and resolutions}\label{s2}

\subsection{Str{\o}mme's embedding} \label{semb}

We begin by describing Str{\o}mme's construction which exhibits the Quot scheme over $\mathbb P^1$ as the zero locus of a regular section of a vector bundle over the product of two Grassmannians \cite {St}.

For each integer $m\geq n$, the embedding takes the form
\[
  \iota_m \colon \quot_{\,\mathbb P^1 }(\mathbb C^N, n)\hookrightarrow {\mathbf G}(V_{m-1}, n)\times {\mathbf G}(V_m, n),
\]
where $$\mathbf G_1={\mathbf G}(V_{m-1}, n), \quad \mathbf G_2={\mathbf G}(V_m, n)$$ are the Grassmannians of $n$-dimensional {\it quotients} of two vector spaces of dimensions $Nm$ and $N(m+1)$ respectively. 
We identify $$V_{m-1}=H^0(\mathbb P^1, \mathcal O_{\mathbb P^1}(m-1)^{\oplus N}), \quad V_m=H^0(\mathbb P^1, \mathcal O_{\mathbb P^1}(m)^{\oplus N}).$$ 
Explicitly, $\iota_m$ is the product of two Grothendieck embeddings. When $m\geq n$, each short exact sequence in the Quot scheme $$0\to S\to \mathbb C^N\otimes \mathcal O_{\mathbb P^1}\to Q\to 0,$$ yields two exact sequences of vector spaces
\[
  0\to H^0(S(m-1))\to H^0(\mathcal O_{\mathbb P^1}(m-1)^{\oplus N})\to H^0(Q(m-1))\to 0,
\]
\[
  0\to H^0(S(m))\to H^0(\mathcal O_{\mathbb P^1}(m)^{\oplus N})\to H^0(Q(m))\to 0.
\]
Then $\iota_m$ is given by the assignment
\[
  \left[0\to S\to \mathbb C^N\otimes \mathcal O_{\mathbb P^1}\to Q\to 0\right]\mapsto \left[V_{m-1}\to H^0(Q(m-1))\right]\times \left[V_m\to H^0(Q(m))\right].
\]

We write $W=H^0(\mathbb P^1, \mathcal O_{\mathbb P^1}(1))$. There is a natural morphism $$V_{m-1}\to V_m\otimes W,$$ obtained from the natural section cutting out the diagonal $\Delta\hookrightarrow \mathbb P^1\times \mathbb P^1$: $$ \mathcal O_{\mathbb P^1\times \mathbb P^1}\to \mathcal O_{\mathbb P^1\times \mathbb P^1}(\Delta)=\mathcal O_{\mathbb P^1}(1)\boxtimes \mathcal O_{\mathbb P^1}(1),$$ tensoring by $\mathcal O_{\mathbb P^1}(m-1)$ on the first factor, and taking cohomology. 

To describe the equations cutting out $\quot_{\,\mathbb P^1 }(\mathbb C^N, n)$ in $\mathbf G_1\times \mathbf G_2$, let 
$$0 \to \mathcal A_1 \to V_{m-1} \otimes \mathcal O_{\mathbf G_1} \to \mathcal B_1\to 0,$$
$$0 \to \mathcal A_2 \to V_m \otimes \mathcal O_{\mathbf G_2} \to \mathcal B_2\to 0$$ 
be the tautological sequences over the two Grassmannians $\mathbf G_1$ and $\mathbf G_2$. 
Let $\rm{pr}_1$ and $\rm{pr}_2$ be the two projections on ${\bf G}_1\times {\bf G}_2$. The sheaf 
\begin{equation}
\label{e}
\mathcal E=\text{pr}_1^\star \mathcal A_1^{\vee}\otimes W \otimes\text{pr}_2^{\star}\,\mathcal B_2\to \mathbf G_1\times \mathbf G_2
\end{equation}
admits a natural section $\sigma$ induced by the composition
\[
  \sigma \colon \text{pr}_1^{\star}\mathcal A_1\to V_{m-1}\otimes \mathcal O_{\mathbf G_1\times \mathbf G_2}\to V_m\otimes W\otimes \mathcal O_{\mathbf G_1\times \mathbf G_2}\to \text{pr}_2^{\star}\, \mathcal B_2 \otimes W.
\]
Str{\o}mme shows that the section $\sigma$ is regular and vanishes exactly along $\quot_{\,\mathbb P^1 }(\mathbb C^N, n)$, see \cite[\S 4]{St}. 

\begin{rmk} \label{rembd}For an arbitrary vector bundle $E\to \mathbb P^1$ and $m$ sufficiently large, we similarly have an embedding
  \[
    \iota \colon \quot_{\mathbb P^1}(E, n)\to {\bf G}_1 \times {\bf G}_2, \quad \mathbf G_1={\mathbf G}(V_{m-1}, n), \quad \mathbf G_2={\mathbf G}(V_m, n)
  \]
  where $$V_{m-1}=H^0(E(m-1)), \quad V_m=H^0(E(m)).$$ 
To obtain a precise lower bound for $m$, note that we need to ensure the vanishing
$$H^1(S(m-1))=H^1(S(m))=0,$$
for every exact sequence $0\to S\to E\to Q\to 0.$ Let $$E=\bigoplus_{i=1}^{N}\mathcal O_{\mathbb P^1}(a_i),$$ and set $a$ to be the largest of the summand degrees $a_i, 1 \leq i \leq N.$
If $$S=\bigoplus_{i=1}^{N} \mathcal O_{\mathbb P^1}(s_i),$$ then we have $s_i \leq a, \, 1 \leq i \leq N,$
since there is an injection
$$
\mathcal O_{\mathbb P^1}(s_i)\to S\to E\to \bigoplus_{i=1}^{N} \mathcal O_{\mathbb P^1}(a).
$$
As $\sum_{i=1}^{N} s_i=\deg E - n,$ we also have 
$$s_i \geq \deg E  - n - (N-1) a, \, \, \, 1 \leq i \leq N.$$ Thus for 
\begin{equation*} \label{bound}
m \geq n + (N-1) a - \deg E,
\end{equation*} 
we obtain $H^1 (S (m-1)) = H^1(S(m))=0$, as wished. 

With no additional conditions on $m$, Str{\o}mme's arguments show that $\quot_{\mathbb P^1}(E, n)$ is also cut out by a regular section $\sigma$ of the tautological bundle \eqref{e} over ${\mathbf G}_1\times {\mathbf G}_2$. Consequently, the arguments below presented for the case of trivial $E$ also carry over without change to arbitrary $E$. 
\end{rmk} 

\subsection{Resolutions}\label{res}
As a result of the above discussion, the section $\sigma$ induces a Koszul resolution
\begin{equation}\label{kos}
  \cdots\to \bigwedge^2\mathcal E^{\vee}\to \bigwedge^{1} \mathcal E^{\vee} \to \mathcal O_{{\mathbf G}_1\times {\mathbf G}_2}\to \mathcal O_{\mathsf{Quot}}\to 0.
\end{equation}
Note that if $\deg L=m$, by the definition of the embedding $\iota_m$ we have
\[
  L^{[n]}=\iota_m^{\star}\,{\rm pr}_2^{\star}\, \mathcal B_2.
\]
Hence, tensoring \eqref{kos} with ${\rm pr}_2^\star \bigwedge^k \cB_2$, we obtain the resolution 
\[
  \cdots\to \bigwedge^2\mathcal E^{\vee}\otimes \text{pr}_2^{\star} \bigwedge^k \mathcal B_2\to \bigwedge^{1} \mathcal E^{\vee} \otimes \text{pr}_2^{\star} \bigwedge^k\mathcal B_2\to  \text{pr}_2^{\star} \bigwedge^k \mathcal B_2\to \left(\iota_m\right)_{\star}\bigg(\bigwedge^k L^{[n]}\bigg)\to 0
\]
over ${\mathbf G_1}\times \mathbf G_2$. We set 
$$
\mathcal V_{\ell} = \bigwedge^\ell\mathcal E^{\vee}\otimes \text{pr}_2^{\star} \bigwedge^k \mathcal B_2, \quad \ell\geq 0.
$$ This yields the resolution 

\begin{equation}\label{resolwedge}
      \cdots \to \mathcal V_{2}\to \mathcal V_{1}\to \mathcal V_0\to \left(\iota_m\right)_{\star}\bigg(\bigwedge^k L^{[n]}\bigg)\to 0,
    \end{equation}
    corresponding to Theorem \ref{ta}\,(1). 
    
    Theorem \ref{ta}\,(2) is conceptually analogous, but the notation becomes slightly more involved. For this reason, we believe it is helpful to present the simpler case (1) first. To treat case (2), we consider the bundle
    \begin{subequations} \begin{equation}\label{bunf} \mathcal F=\left(\bigwedge^{p_1} L^{[n]}\right)^{\vee}\otimes \cdots \otimes \left(\bigwedge^{p_t} L^{[n]}\right)^{\vee}\otimes \left(\bigwedge^{k} L^{[n]}\right).\end{equation} By similar reasoning, we are led to the resolution 
    \begin{equation}\label{resolut}
  \cdots\to \mathcal X_2\to \mathcal X_1\to \mathcal X_0\to \left(\iota_{m}\right)_{\star} \mathcal F\to 0,\end{equation}
  \end{subequations}
where
\[
  \quad \mathcal X_{\ell}=\bigwedge^{\ell} \mathcal E^{\vee} \otimes  \bigg({\rm{pr}}_2^{\star} \bigwedge^{p_1} \mathcal B_2^{\vee}\otimes {\rm{pr}}_2^{\star} \bigwedge^{p_2} \mathcal B_2^{\vee} \otimes \cdots \otimes {\rm{pr}}_2^{\star} \bigwedge^{p_t} \mathcal B_2^{\vee}\otimes \text{pr}_2^{\star} \bigwedge^{k} \mathcal B_2\bigg). \]
  
\begin{proposition}\label{resolwedgep}
  \begin{subeqns}
    In the resolution \eqref{resolwedge}
    \begin{equation*}
      \cdots \to \mathcal V_{2}\to \mathcal V_{1}\to \mathcal V_0\to \left(\iota_m\right)_{\star}\bigg(\bigwedge^k L^{[n]}\bigg)\to 0,
    \end{equation*}
    the sheaves $\mathcal V_{\ell}$ have no cohomology for $\ell\geq 1$, while the sheaf $\mathcal V_0$ has no higher cohomology. 
    
    More generally, in the resolution \eqref{resolut}, the sheaves $\mathcal X_\ell$ have no cohomology for $\ell\geq 1$, while the sheaf $\mathcal X_0$ has no higher cohomology. 
  \end{subeqns}
\end{proposition}
  
For the symmetric products, we similarly define
\[
  \mathcal W_{\ell}=\bigwedge^\ell\mathcal E^{\vee}\otimes \text{pr}_2^{\star} \,\Sym^k \mathcal B_2, 
\]
and have an analogous resolution \begin{equation}\label{resolsymr}
  \cdots \to \mathcal W_{2}\to \mathcal W_{1}\to \mathcal W_0\to \left(\iota_m\right)_{\star}\left(\Sym^k L^{[n]}\right)\to 0. 
\end{equation} This corresponds to Theorem \ref{tb}\,(1). For the more general Theorem \ref{tb}\,(2), we let 
\begin{subequations}
\begin{equation}\label{bung}\mathcal G=\left(\bigwedge^{p_1} L^{[n]}\right)^{\vee}\otimes \cdots \otimes \left(\bigwedge^{p_t} L^{[n]}\right)^{\vee}\otimes \text{Sym}^{k} L^{[n]}.\end{equation} Setting \[
  \quad \mathcal Y_{\ell}=\bigwedge^{\ell} \mathcal E^{\vee} \otimes  \bigg({\rm{pr}}_2^{\star} \bigwedge^{p_1} \mathcal B_2^{\vee}\otimes {\rm{pr}}_2^{\star} \bigwedge^{p_2} \mathcal B_2^{\vee} \otimes \cdots \otimes {\rm{pr}}_2^{\star} \bigwedge^{p_t} \mathcal B_2^{\vee}\otimes \text{pr}_2^{\star}\, \text{Sym}^{k} \mathcal B_2\bigg), \] we obtain the resolution \begin{equation}\label{resolus}
  \cdots\to \mathcal Y_2\to \mathcal Y_1\to \mathcal Y_0\to \left(\iota_{m}\right)_{\star} \mathcal G\to 0.\end{equation}
  \end{subequations}
\begin{proposition}\label{resolsymp}
  \begin{subeqns}
    In the resolution \eqref{resolsymr} 
  \begin{equation*}
  \cdots \to \mathcal W_{2}\to \mathcal W_{1}\to \mathcal W_0\to \left(\iota_m\right)_{\star}\left(\Sym^k L^{[n]}\right)\to 0,
\end{equation*}
the sheaves $\mathcal W_{\ell}$ have no cohomology if $\ell\geq 1$ and $\deg L\geq n\geq k$, while $\mathcal W_0$ has no higher cohomology. 

More generally, under the same assumptions, in the resolution \eqref{resolus}, the sheaves $\mathcal Y_\ell$ have no cohomology for $\ell\geq 1$, while the sheaf $\mathcal Y_0$ has no higher cohomology. 
\end{subeqns}
\end{proposition}

A further analysis is needed for Theorem \ref{tc}. The case $L=M$ is already covered either by Theorem \ref{ta}\,(2) or Theorem \ref{tb}\,(2) for $k=0$. Thus we may assume $\deg M=m\geq n$ and $\deg L=m-1$. In this case, we have $$L^{[n]}=\iota_m^{\star} \,{\rm {pr}}_1^{\star}\, \mathcal B_1,\quad M^{[n]}=\iota_m^{\star} \,\rm{pr}_2^{\star}\,\mathcal B_2.$$
Thus for the bundle 
\[
  \mathcal H=\bigg(\bigwedge^{p_1} L^{[n]}\bigg)^{\vee}\otimes \bigg(\bigwedge^{p_2} M^{[n]}\bigg)^{\vee}\otimes \cdots \otimes \bigg(\bigwedge^{p_t} M^{[n]}\bigg)^{\vee}
\]
which appears in Theorem \ref{tc}, we obtain a resolution
\begin{equation}\label{resolu}
  \cdots\to \mathcal U_2\to \mathcal U_1\to \mathcal U_0\to \left(\iota_{m}\right)_{\star} \mathcal H\to 0,\end{equation}
where
\[
  \quad \mathcal U_{\ell}=\bigwedge^{\ell} \mathcal E^{\vee} \otimes  \bigg({\rm{pr}}_1^{\star}\bigwedge^{p_1} \mathcal B_1^{\vee}\otimes {\rm{pr}}_2^{\star} \bigwedge^{p_2}\mathcal B_2^{\vee}\otimes \cdots \otimes {\rm{pr}}_2^{\star} \bigwedge^{p_t} \mathcal B_2^{\vee}\bigg).
\]

\begin{proposition}\label{resolup} For $p_1, \ldots, p_t$ not all zero, the cohomology of $\mathcal U_{\ell}$ vanishes for all $\ell\geq 0$. 
\end{proposition}

\subsubsection{The main theorems} \label{rmn}Before turning our attention to the proofs of the above propositions, we note that our main Theorems \ref{ta}, \ref{tb} and \ref{tc} follow immediately from them. \vskip.1in

For Theorem \ref{ta}, we use Proposition \ref{resolwedgep}. To establish case (1) of the theorem, we make use of the resolution \eqref{resolwedge}. The associated spectral sequence shows that the higher cohomology of $\bigwedge^k L^{[n]}$ vanishes, while in degree zero we have 
\[
  H^0\bigg(\mathsf{Quot}_{\,\mathbb P^1}(\mathbb C^N, n), \bigwedge^k L^{[n]}\bigg)=H^0(\mathbf G_1\times \mathbf G_2, \mathcal V_0).
\]
Recalling that $\mathcal V_0=\text{pr}_2^{\star} \, \bigwedge^k\mathcal B_2$, we compute 
\begin{align*}
  H^0(\mathbf G_1\times \mathbf G_2, \mathcal V_0)=H^0\bigg(\mathbf G_2, \bigwedge^k \mathcal{B}_2\bigg)=\bigwedge^k V_m=\bigwedge^k H^0(\mathbb C^N \otimes \mathcal O_{\mathbb P^1}(m))=\bigwedge^k H^0(L^{\oplus N}),
\end{align*} as needed. 

For part (2) of Theorem \ref{ta}, we note that $$\textnormal{Ext}^{i} \bigg(\bigwedge^{p_1} L^{[n]}\otimes \cdots \otimes \bigwedge^{p_t} L^{[n]}, \bigwedge^{k} L^{[n]}\bigg)=H^i(\mathcal F),$$ where the bundle $\mathcal F$ is defined in \eqref{bunf}. Using the resolution \eqref{resolut} and Proposition \ref{resolwedgep}, we see that the only contribution to the cohomology of $\mathcal F$ comes from the term \begin{align*}H^0({\bf G}_1\times {\bf G}_2, \mathcal X_0)=H^0\bigg ({\bf G}_2, \bigwedge^{p_1}\mathcal B_2^{\vee}\otimes \cdots \otimes \bigwedge^{p_t} \mathcal B_2^{\vee}\otimes \bigwedge^k \mathcal B_2\bigg)=\bigwedge^{k - |p|} H^0 ({\bf G}_2, \mathcal B_2)=\bigwedge^{k- |p|} H^0(L^{\oplus N})\end{align*} for $|p|=p_1+\cdots+p_t\leq k\leq n$. The second equality requires further explanation. Since we need additional notation, the argument will be presented later, see equation \eqref{boomboom} in the proof of Proposition \ref{resolwedgep} below. \vskip.1in

Turning to Theorem \ref{tb}, for part (1), we make use of the resolution \eqref{resolsymr} and Proposition \ref{resolsymp}. This time, the initial term $\mathcal W_0=\text{pr}_2^{\star}\, \Sym^k \mathcal B_2$ has sections
$$
H^0(\mathbf G_1\times \mathbf G_2, \mathcal W_0)=H^0(\mathbf G_2, \Sym^k \cB_2)=\Sym^k V_m=\Sym^k H^0(\mathbb C^N \otimes \mathcal O_{\mathbb P^1}(m))=\Sym^k H^0(L^{\oplus N}).
$$ For part (2), we note $$\textnormal{Ext}^{i} \bigg(\bigwedge^{p_1} L^{[n]}\otimes \cdots \otimes \bigwedge^{p_t} L^{[n]}, \text{Sym}^{k} L^{[n]}\bigg)=H^i(\mathcal G),$$ where the sheaf $\mathcal G$ was defined in \eqref{bung}. 
Using the resolution \eqref{resolus} and Proposition \ref{resolsymp}, the only nontrivial contribution to the cohomology of $\mathcal G$ comes from the sheaf $\mathcal Y_0$ and it equals 
 \begin{align*}H^0\bigg ({\bf G}_2, \bigwedge^{p_1}\mathcal B_2^{\vee}\otimes \cdots \otimes \bigwedge^{p_t} \mathcal B_2^{\vee}\otimes \text{Sym}^k \mathcal B_2\bigg)=\text{Sym}^{k - |p|} H^0 ({\bf G}_2, \mathcal B_2)=\text{Sym}^{k- |p|} H^0(L^{\oplus N}).\end{align*} This requires that all $p_j\in \{0, 1\}$ and $|p|\leq k$. The cohomology vanishes altogether if this condition fails. The first equality will be explained after we develop more notation, see equation \eqref{boom2} in the proof of Proposition \ref{resolsymp} below.
 \vskip.1in
Finally, for Theorem \ref{tc}, the argument uses Proposition \ref{resolup}. This time around, the cohomology vanishes for all terms of the resolution \eqref{resolu} and in all degrees. \qed

\subsection {Analysis of the resolutions}\label{ss2} We now turn to Propositions \ref{resolwedgep}, \ref{resolsymp}, \ref{resolup} and deduce them from the Grassmannian vanishing results of Section \ref{s3}. We begin by making the terms of the resolutions more explicit. 

\subsubsection{Partitions and Cauchy's formula}

We use standard terminology on partitions $\lambda=(\lambda_1, \ldots, \lambda_r),$ where $\lambda_1\geq \lambda_2\geq \ldots \geq \lambda_r\geq 0.$ We set $$|\lambda|=\sum_{i=1}^{r} \lambda_i,$$ and we let $\lambda^{\dagger}$ be the transpose partition obtained by exchanging the rows and columns of the Young diagram of $\lambda$.

Throughout, we will always assume that our base scheme is defined over a field of characteristic $0$. 

For each partition $\lambda$, we let ${\mathbf S}_{\lambda}$ denote the associated Schur functor (for example, these are defined in \cite[\S 2.1]{W} where they are called $L_{\lambda^\dagger}$). For a partition $\lambda$ and any vector bundle $V\to Y$ over a base $Y$, there is an associated vector bundle ${\mathbf S}_{\lambda}(V)\to Y.$ The cases $\lambda=(1^k)$ and $\lambda=(k)$ correspond to the $k$th exterior and $k$th symmetric powers, respectively. 

The vector bundles $\bS_{\lambda}(V)\to Y$ are also defined when $\lambda$ is not a partition but rather an arbitrary dominant weight $\lambda_1\geq \cdots \geq \lambda_r$, where $r = \rank(V)$, and we now allow the entries to be negative. We have
\[
  \bS_{-\lambda}(V)=\bS_{\lambda}(V^{\vee}),
\]
where $-\lambda$ denotes the sequence $-\lambda_r \ge \cdots \ge -\lambda_1$. In addition, $$\bS_{\lambda}(V)\otimes \det V=\bS_{\lambda+(1^r)} (V).$$

If $V, W\to Y$ are two vector bundles, Cauchy's identity $$\bigwedge^\ell (V\otimes W)=\bigoplus_{|\lambda|=\ell} {\mathbf S}_{\lambda^\dagger} (V)\otimes  {\mathbf S}_{\lambda} (W)$$ holds, where the sum is over all partitions $\lambda$ of size $\ell$. We only need to consider those partitions $\lambda$ with at most $\rank(W)$ rows and at most $\rank(V)$ columns since the term is 0 otherwise. Applying this formula to the bundle $\mathcal E$ whose section cuts out the Quot scheme, we obtain
\begin{subequations}
  \begin{equation}\label{cauchy}
  \bigwedge^\ell \mathcal E^{\vee}=\bigoplus_{|\lambda|=\ell} {\mathbf S}_{\lambda^\dagger} \left(\mathcal A_1\right)\boxtimes {\mathbf S}_{\lambda} \left(\mathcal B_2^{\vee}\oplus \mathcal B_2^{\vee}\right).
\end{equation}
Here, $\lambda$ is a partition with at most $2n$ rows and the number of columns at most equal to
\begin{equation}\label{size}
  \rank(\cA_1) = \dim V_{m-1}-n=\dim V_m-N-n\leq \dim V_m - n - 1.
\end{equation}
\end{subequations}
In the discussion below, the abbreviation $d=\dim V_m$ will often be used. 

\medskip

Using \eqref{cauchy}, we immediately obtain the expressions
\begin{subequations}
  \begin{equation}\label{vlcauchy}
  \mathcal V_{\ell}=\bigoplus_{|\lambda|=\ell} {\mathbf S}_{\lambda^{\dagger}} \left(\mathcal A_1\right)\boxtimes \bigg({\mathbf S}_{\lambda} \left(\mathcal B_2^{\vee}\oplus \mathcal B_2^{\vee}\right)\otimes \bigwedge^k \mathcal B_2\bigg),
\end{equation}
and
\begin{equation}\label{wlcauchy}
  \mathcal W_{\ell}=\bigoplus_{|\lambda|=\ell} {\mathbf S}_{\lambda^{\dagger}} \left(\mathcal A_1\right)\boxtimes \left({\mathbf S}_{\lambda} \left(\mathcal B_2^{\vee}\oplus \mathcal B_2^{\vee}\right)\otimes \Sym^k \mathcal B_2\right) 
\end{equation} corresponding to Theorem \ref{ta}\,(1) and Theorem \ref{tb}\,(1). For the second halves of the two theorems, the expressions are slightly more complicated due to additional wedge powers: \begin{equation}\label{c9} 
\mathcal X_{\ell}=\bigoplus_{|\lambda|=\ell} {\mathbf S}_{\lambda^{\dagger}} \left(\mathcal A_1\right)\boxtimes \bigg({\mathbf S}_{\lambda} \left(\mathcal B_2^{\vee}\oplus \mathcal B_2^{\vee}\right)\otimes \bigwedge^{p_1}\mathcal B_2^{\vee}\otimes \cdots \otimes \bigwedge^{p_t} \mathcal B_2^{\vee}\otimes \bigwedge^k \mathcal B_2\bigg)\,.
\end{equation}
and \begin{equation}\label{c95} 
\mathcal Y_{\ell}=\bigoplus_{|\lambda|=\ell} {\mathbf S}_{\lambda^{\dagger}} \left(\mathcal A_1\right)\boxtimes \bigg({\mathbf S}_{\lambda} \left(\mathcal B_2^{\vee}\oplus \mathcal B_2^{\vee}\right)\otimes \bigwedge^{p_1}\mathcal B_2^{\vee}\otimes \cdots \otimes \bigwedge^{p_t} \mathcal B_2^{\vee}\otimes \text{Sym}^k \mathcal B_2\bigg)\,.
\end{equation} Finally, we have 
\begin{equation}\label{ulcauchy2}
  \mathcal U_{\ell}=\bigoplus_{|\lambda|=\ell} \bigg({\mathbf S}_{\lambda^{\dagger}} (\mathcal A_1)\otimes  \bigwedge^{p_1} \mathcal B_1^{\vee}\bigg)\boxtimes \bigg({\mathbf S}_{\lambda} (\mathcal B_2^{\vee}\oplus \mathcal B_2^{\vee})\otimes \bigwedge^{p_2} \mathcal B_2^{\vee} \otimes \cdots \otimes \bigwedge^{p_t} \mathcal B_2^{\vee}\bigg).
\end{equation}
\end{subequations}
\subsubsection{The Borel--Weil--Bott theorem} \label{bwbs}To compute the cohomology of the above bundles, we use the Borel--Weil--Bott theorem \cite {B}. For integers $0< n < d$, let $\mathbf G=\mathbf G(d, n)$ denote the Grassmannian of $n$-dimensional {\it quotients} of a $d$-dimensional vector space, and let $\mathcal A, \mathcal B\to \mathbf G$ denote the tautological subbundle and quotient. For a partition $$\mu=(\mu_1, \ldots, \mu_{d-n})$$ with $d-n$ rows, we form the string
\begin{equation}\label{string1}
  \rho+(0,\mu)=(d-1, d-2, \ldots, 1, 0)+(\underbrace{0, \ldots, 0}_n, \mu_1, \ldots, \mu_{d-n}).
\end{equation}

\begin{theorem}[Borel--Weil--Bott] \label{bwb} The bundle ${\mathbf S}_{\mu} \left(\mathcal A\right)$ has at most one non-zero cohomology group. Furthermore, if the string $\rho+(0, \mu)$ contains repetitions, then all cohomology groups of ${\mathbf S}_{\mu}\left(\mathcal A\right)$ vanish. 
\end{theorem} 

\begin{subequations}
 This formulation can be found in \cite [Corollary 4.1.9]{W}. We are using a few translations. First, the Weyl functors $K_\gamma$ defined there are isomorphic to the Schur functors used here, see \cite[Proposition 2.1.18(c)]{W}. Moreover, the dual of the tautological subbundle $\cA$ is the quotient bundle on the dual Grassmannian, and on the latter space, ${\bf S}_\mu (\cA)$ corresponds to $\mathcal{V}(0,\mu)$ in the notation of \cite{W}.

  We note from Theorem \ref{bwb} that ${\mathbf S}_{\mu} \left(\mathcal A\right)$ has no cohomology provided there exists $j$ such that
\begin{equation}\label{condition}
  j\leq \mu_j\leq n+j-1.
\end{equation} 

Let us record the ``dual'' rephrasing of condition \eqref{condition} which is also useful here. For a partition $\nu$ with $n$ rows, the bundle ${\mathbf S}_{\nu} \left(\mathcal B^{\vee}\right)$ over the Grassmannian $\mathbf G(d, n)$ has no cohomology provided that there exists $j$ such that 
\begin{equation} \label{conditiondual0}
j\leq \nu_j\leq d-n+j-1.
\end{equation} 
\end{subequations}

In Proposition \ref{resolup}, we also consider the bundle $$\mathbf S_{\mu}\left(\mathcal A\right)\otimes \bigwedge^{p} \mathcal B^{\vee}=\mathbf S_{-\mu}\left(\mathcal A^{\vee}\right)\otimes {\mathbf S}_{(1^{p})} \mathcal B^{\vee}$$ for $0\leq p\leq n$. Again by the Borel--Weil--Bott theorem \cite [Corollary 4.1.9]{W}, all cohomology vanishes provided that the string
\begin{subequations}
\begin{equation}\label{string2}
  (d-1, \ldots, 0) + (-\mu_{d-n}, \ldots, -\mu_{1}, \underbrace{1, \ldots, 1}_{p}, \underbrace{0, \ldots, 0}_{n-p})
\end{equation}
has repetitions. That happens when there exists $j$ such that 
\begin{equation}
\label{conditionplus}
j-1\leq \mu_j\leq n+j-1 \text { and } \mu_j\neq j+p-1. 
\end{equation}
\end{subequations}
Of course, the case $p=0$ recovers \eqref{condition}. In fact, for $p=0$, the string \eqref{string1} has repetitions if and only if the same holds for \eqref{string2}.

\subsubsection{Indices of partitions} The following definition is not standard but is crucial for our arguments. 

\begin{definition}\label{indef}
  Let $n$ be a non-negative integer. Let $\lambda\neq 0$ be a partition satisfying the following condition
\begin{itemize} 
\item [$(*)$]for all $j$, the number of boxes in the $j{\text{th}}$ column of $\lambda$ is either $<j$ or $\geq n+j$. 
\end{itemize} 
Let $i$ denote the largest index $j$ such that the $j{\text{th}}$ column has $\geq n+j$ boxes. We refer to $i$ as the {\bf $n$-index} of $\lambda$. If $\lambda$ does not satisfy $(*)$, we leave the $n$-index undefined. 
\end{definition}

\begin{figure}
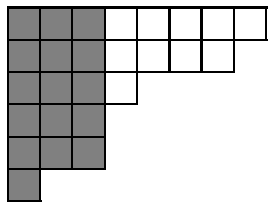

\ytableausetup{centertableaux, boxsize=1em}
\begin{ytableau}
*(gray) & *(gray) & *(gray)  &*(white)& *(white) & *(white) & *(white) & *(white)\\
*(gray)  & *(gray) & *(gray) & *(white) & *(white) & *(white) & *(white)\\
*(gray) & *(gray) & *(gray) & *(white) \\
*(gray) & *(gray) & *(gray) \\
*(gray) & *(gray) & *(gray)  \\
*(gray)
\end{ytableau}
\caption{A partition of $2$-index $i=3$}
\label{part}
\end{figure}

It may help to visualize partitions $\lambda$ of $n$-index $i$. There are $i$ ``long" columns with at least $n+i$ boxes, while the remaining columns are ``short" containing at most $i$ boxes. In Figure \ref{part}, the long columns are shown in gray, while the short columns are white. Thus, for a partition $\lambda$ of $n$-index $i$, we have
\begin{equation}\label{indin}
  \lambda_{i+1}=\cdots=\lambda_{i+n}=i.
\end{equation}

The following variation is needed for Proposition \ref{resolup} and is connected to condition \eqref{conditionplus} above. 

\begin{definition} \label{def:pnindex}
  Let $0\leq p\leq n$ be integers. Let $\lambda\neq 0,$ $\lambda\neq (1^p)$ be a partition satisfying the following condition
\begin{itemize} 
\item [$(**)$]for all $j$, the number of boxes in the $j{\text{th}}$ column of $\lambda$ is either $<j-1$ or $\geq n+j$ or equal to $j+p-1$. 
\end{itemize} 
Let $i$ denote the largest index $j$ such that the $j{\text{th}}$ column has $\geq n+j$ boxes. We refer to $i$ as the {\bf $(p, n)$-index} of $\lambda$, when defined. The case $p=0$ corresponds to the $n$-index defined above. 
\end{definition}

The partition $\lambda=(1^p)$ is not considered here. The reason is that $\lambda$ satisfies $(**)$, yet no column has $\geq n+1$ boxes, so the index is undefined. 
 
 \begin{figure}
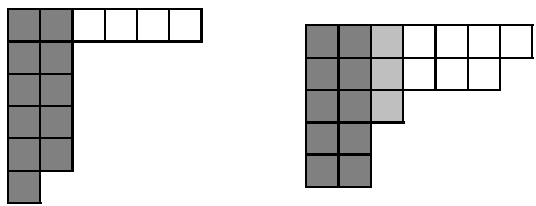

\ytableausetup{centertableaux}
\begin{ytableau}
*(gray) & *(gray)   & *(white)& *(white) & *(white) & *(white)\\
*(gray)  & *(gray)  \\
*(gray) & *(gray)\\
*(gray) & *(gray)  \\
*(gray) & *(gray) \\
*(gray)
\end{ytableau}
\quad\quad \quad
\begin{ytableau}
*(gray) & *(gray)  & *(lightgray) & *(white)& *(white) & *(white) & *(white)\\
*(gray)  & *(gray)  & *(lightgray) & *(white) & *(white) & *(white)\\
*(gray) & *(gray) & *(lightgray)\\
*(gray) & *(gray)  \\
*(gray) & *(gray) 
\end{ytableau}
\caption{Partitions of $(1, 3)$-index $i=2$}
\label{part2}
\end{figure}
 
 For a partition $\lambda$ of $(p, n)$-index $i$, two shapes are possible:
\begin{itemize}
\item [(a)] the partition $\lambda$ has $i$ ``long'' columns with $\geq n+i$ boxes, and the remaining columns are ``short'' having $\leq i-1$ boxes. 
\item [(b)] the partition $\lambda$ has $i$ ``long'' columns with $\geq n+i$ boxes, the $(i+1){\text{st}}$ column has $p+i$ boxes, and the remaining columns are ``short'' having $\leq i$ boxes. 
\end{itemize}
The partitions in Figure \ref{part2} satisfy (a) and (b) respectively. For case (b), the long and short columns are shown in dark gray and white, while the middle $(i+1){\text{st}}$ column is lighter gray. In both cases
\begin{equation}
  \label{ineqb}
  i\leq \lambda_{i+1}\leq i+1, \,\,\ldots,\,\, i\leq \lambda_{i+n}\leq i+1.
\end{equation}

\begin{proof}[Proof of Proposition \ref{resolwedgep}] For simplicity, we consider the case of the resolution $\mathcal V_{\bullet}$ first. Recall from \eqref{vlcauchy} that $$\mathcal V_{\ell}=\bigoplus_{|\lambda|=\ell} {\mathbf S}_{\lambda^{\dagger}} \left(\mathcal A_1\right)\boxtimes \bigg({\mathbf S}_{\lambda} \left(\mathcal B_2^{\vee}\oplus \mathcal B_2^{\vee}\right)\otimes \bigwedge^k \mathcal B_2\bigg).$$

For $\ell=0$, we must have $\lambda=0$, and $\mathcal V_0=\text{pr}^{\star} \bigwedge^k \mathcal B_2$ has no higher cohomology. 

When $\ell\geq 1$, we have $\lambda\neq 0$. For a partition $\lambda\neq 0$ appearing in the above sum, we distinguish two mutually exclusive situations:
\begin{itemize}
\item [$(\dagger)$] there exists $j$ such that $j\leq \lambda^{\dagger}_j\leq n+j-1$, or 
\item [$(*)$] for all $j$, the number of boxes in the $j{\text{th}}$ column of $\lambda$ is either $<j$ or $\geq n+j$. 
\end{itemize}
Of course, condition $(*)$ already appeared in Definition \ref{indef}. 
In case $(\dagger)$, we noted in \eqref {condition} that ${\mathbf S}_{\lambda^{\dagger}}\left(\mathcal A_1\right)$ has no cohomology. In case $(*)$, Proposition \ref{vanishwedge} in Section \ref{s3} below shows that $${\mathbf S}_{\lambda} \left(\mathcal B_2^{\vee}\oplus \mathcal B_2^{\vee}\right)\otimes \bigwedge^k \mathcal B_2$$ has no cohomology either. Consequently, $\mathcal V_{\ell}$ has no cohomology when $\ell\geq 1$, establishing the first half of Proposition \ref{resolwedgep}. 

For the second half, recall from \eqref{c9} that \[\mathcal X_{\ell}=\bigoplus_{|\lambda|=\ell} {\mathbf S}_{\lambda^{\dagger}} \left(\mathcal A_1\right)\boxtimes \bigg({\mathbf S}_{\lambda} \left(\mathcal B_2^{\vee}\oplus \mathcal B_2^{\vee}\right)\otimes \bigwedge^{p_1}\mathcal B_2^{\vee}\otimes \cdots \otimes \bigwedge^{p_t} \mathcal B_2^{\vee}\otimes \bigwedge^k \mathcal B_2\bigg)\,.
  \] When $\ell\neq 0$, then either $\lambda\neq 0$ satisfies condition $\left(\dagger\right)$ which guarantees cohomology vanishing for the Schur bundle on the first factor, or else $\lambda\neq 0$ satisfies $(*)$. The latter situation yields vanishing on the second factor by Proposition \ref{vanishdual}\,(1) below. The hypothesis of the Proposition is verified since $\lambda$ was seen in \eqref{size} to have at most $$\dim V_m-N-n\leq \dim V_m- (t+1)-n$$ columns and $t\leq N-1$.

In the proof of Theorem \ref{ta} in Section \ref{rmn}, we also claimed that the cohomology of $\mathcal X_0$ is given by \begin{align}\label{boomboom} H^0\bigg ({\bf G}_2, \bigwedge^{p_1}\mathcal B_2^{\vee}\otimes \cdots \otimes \bigwedge^{p_t} \mathcal B_2^{\vee}\otimes \bigwedge^k \mathcal B_2\bigg)=H^0 \bigg({\bf G}_2, \bigwedge^{k- |p|} \mathcal B_2\bigg)=\bigwedge^{k- |p|} H^0 ({\bf G}_2, \mathcal B_2).\end{align} Here we assume $|p|=p_1+\cdots+p_t\leq k\leq n$, while the answer is understood to be $0$ if this assumption fails. We can now justify the first equality using Pieri's rule combined with Borel--Weil--Bott. This is certainly well-known, but it appears easier to write the argument than to find a reference. 

Consider an arbitrary Grassmannian ${\bf G}(d, n)$ with tautological quotient $\mathcal B$, and assume $t\leq d-n$. The latter is true in our setting since $t\leq N-1\leq \dim V_m-n$. We inspect the tensor product \begin{equation}\label{tp}\bigwedge^{p_1}\mathcal B^{\vee}\otimes \cdots \otimes \bigwedge^{p_t} \mathcal B^{\vee}\otimes \bigwedge^k \mathcal B=\bigwedge^{p_1}\mathcal B^{\vee}\otimes \cdots \otimes \bigwedge^{p_t} \mathcal B^{\vee}\otimes \bigwedge^{n-k} \mathcal B^{\vee} \otimes \det \mathcal B.\end{equation} By Pieri's rule, tensorization by $\bigwedge^p \mathcal B^{\vee}$ has the effect of adding $p$ boxes, no two in the same row. Thus, if ${\bf S}_{\nu}(\mathcal B^{\vee})$ is any summand of $\bigwedge^{p_1}\mathcal B^{\vee}\otimes \cdots \otimes \bigwedge^{p_t} \mathcal B^{\vee}\otimes \bigwedge^{n-k} \mathcal B^{\vee}$, then we create at most $t+1$ columns. Hence $$1\leq \nu_1\leq t+1\leq d-n+1.$$ 
On the other hand, tensorization by $\det \mathcal B$ subtracts $1$ box from all the rows. Consequently, the summands ${\bf S}_{\mu}(\mathcal B^{\vee})$ that appear in \eqref{tp} satisfy $$\mu_1=\nu_1-1.$$ In general, we have $1\leq \mu_1\leq d-n$, so Borel--Weil--Bott shows that ${\bf S}_{\mu}(\mathcal B^{\vee})$ has no cohomology, see condition \eqref{conditiondual0}. There is one exception  corresponding to $\mu_1=0$. In this case, we must have $\nu_1=1$, which forces $\nu=(1^{|p|+n-k})$ and then $-\mu=(1^{k-|p|})$. This yields the term ${\bf S}_{\mu}(\mathcal B^{\vee})={\bf S}_{-\mu} (\mathcal B)=\bigwedge^{k-|p|} \mathcal B$, and justifies \eqref{boomboom}. 
\end{proof} 
\vskip.1in
\begin{proof}[Proof of Proposition \ref{resolsymp}] We consider the simpler case of the resolution $\mathcal W_{\bullet}$ first. By \eqref{wlcauchy} we have $$\mathcal W_{\ell}=\bigoplus_{|\lambda|=\ell} {\mathbf S}_{\lambda^{\dagger}} \left(\mathcal A_1\right)\boxtimes \left({\mathbf S}_{\lambda} \left(\mathcal B_2^{\vee}\oplus \mathcal B_2^{\vee}\right)\otimes \Sym^k \mathcal B_2\right).$$ 
The argument is similar to that of Proposition \ref{resolwedge}. This time, to deal with case $(*)$ we invoke Proposition \ref{vanishsym}. 

The analysis of the bundles $\mathcal Y_{\ell}$ for $\ell\geq 1$ in the general case uses Proposition \ref{vanishdual}\,(2) instead.  

In Section \ref{rmn}, we also needed the cohomology of the bundle $\mathcal Y_0$. To this end, we show that on an arbitrary Grassmannian ${\bf G}(d, n)$, for $t\leq d-n$ and positive integers $p_1, \ldots, p_t>0$, the bundle $\bigwedge^{p_1} \mathcal B^{\vee}\otimes \ldots \otimes \bigwedge^{p_t} \mathcal B^{\vee}\otimes \Sym^k \mathcal B$ has no cohomology unless $p_1=\cdots=p_t=1$ and $k\geq t$, in which case there is only one nontrivial cohomology group 
\begin{equation}\label{boom2}H^0\left(\bigwedge^{p_1} \mathcal B^{\vee}\otimes \cdots \otimes \bigwedge^{p_t} \mathcal B^{\vee}\otimes \Sym^k \mathcal B\right)=\text{Sym}^{k-t} H^0(\mathcal B).\end{equation} Indeed, let ${\bf S}_{\mu}(\mathcal B^{\vee})$ be a summand of $\bigwedge^{p_1} \mathcal B^{\vee}\otimes \cdots \otimes \bigwedge^{p_t} \mathcal B^{\vee}\otimes \Sym^k \mathcal B$. Dualizing, we have that ${\bf S}_{-\mu}(\mathcal B^{\vee})$ is a summand of
$$\bigwedge^{p_1} \mathcal B\otimes \cdots \otimes \bigwedge^{p_t} \mathcal B\otimes \Sym^k \mathcal B^{\vee}=\bigwedge^{n-p_1} \mathcal B^{\vee}\otimes \cdots \otimes \bigwedge^{n-p_t} \mathcal B^{\vee}\otimes \Sym^k \mathcal B^{\vee} \otimes (\det \mathcal B)^{t}.$$
Let ${\bf S}_{\nu}(\mathcal B^{\vee})$ be a summand of $\bigwedge^{n-p_1} \mathcal B^{\vee}\otimes \cdots \otimes \bigwedge^{n-p_t} \mathcal B^{\vee}\otimes \Sym^k \mathcal B^{\vee}$. Applying the Pieri rules, we start with the partition $(k)$ corresponding to $\text{Sym}^k \,\mathcal B^{\vee}$, to which we add $n-p_1, \ldots, n-p_t$ boxes respectively, such that at each stage we do not add two boxes to the same row. Therefore $$0\leq \nu_n\leq t$$ (unless $n=1$ which can be treated separately). Furthermore, if $\nu_n=t$, then $\nu_2=\cdots=\nu_{n-1}=t$. In this case we can say a bit more. Note that the last $n-1$ rows of $\nu$ each contain $t$ boxes, and $n-p_1, \ldots, n-p_t$ are all $\leq n-1$. For this to be possible, equality must hold, hence $p_1=\ldots=p_t=1$. Moreover, all boxes have to be added to the last $n-1$ rows. In particular, no boxes can be added to the first row. Hence $\nu$ is the partition $$\nu_1=k, \,\,\nu_2=\ldots=\nu_n=t.$$ For this to make sense, we must have $\nu_1=k\geq t=\nu_2$. 
%Counting the total number of boxes used, we find $\nu_1$ is obtained by adding $$(n-p_1)+\cdots+(n-p_t) - (n-1)t=t-|p|$$ boxes to the first row, which has $k$ boxes. For this to make sense, we need $|p|\leq t$, or equivalently $p_1=\cdots=p_t=1$ and $\nu_1=k+(t-|p|)=k$.  

Finally, each ${\bf S}_{-\mu} (\mathcal B^{\vee})={\bf S}_{\nu}(\mathcal B^{\vee})\otimes (\det \mathcal B)^t$ satisfies $\mu_{n+1-i}+\nu_i=t.$ Since $0\leq \nu_n\leq t$, we also have $0\leq \mu_1\leq t.$ Since $t\leq d-n$, it follows by condition \eqref{conditiondual0} that we only have nontrivial cohomology for ${\bf S}_{\mu} (\mathcal B^{\vee})$ when $\mu_1=0$. This case corresponds to $\nu_n=t$. We have seen above this means $\nu_1=k$ and $\nu_2=\cdots=\nu_n=t$. This yields $\mu=(0, \ldots, 0, t-k)$ and ${\bf S}_{\mu} (\mathcal B^{\vee}) = \text{Sym}^{k-t} \mathcal B$, which implies the claim. \end{proof}
\vskip.1in
\begin{proof}[Proof of Proposition \ref{resolup}.]
 We need to inspect
\[
  \mathcal U_{\ell}=\bigoplus_{|\lambda|=\ell} \bigg(\mathbf S_{\lambda^{\dagger}}\left(\mathcal A_1\right)\otimes \bigwedge^{p_1} \mathcal B_1^{\vee}\bigg)\boxtimes \bigg({\mathbf S}_{\lambda} (\mathcal B_2^{\vee}\oplus \mathcal B_2^{\vee})\otimes \bigwedge^{p_2} \mathcal B_2^{\vee} \otimes \cdots \otimes \bigwedge^{p_t} \mathcal B_2^{\vee}\bigg).
\]
The case $\ell=0$ follows from either \eqref{boomboom} or \eqref{boom2} with $k=0$. 

Let $\ell\neq 0$. When $\lambda^{\dagger}$ satisfies \eqref{conditionplus} (for $p=p_1$), the first factor $\mathbf S_{\lambda^{\dagger}}\left(\mathcal A_1\right)\otimes \bigwedge^{p_1} \mathcal B_1^{\vee}$ has no cohomology. Otherwise, $\lambda$ satisfies condition $(**)$ from Definition~\ref{def:pnindex}. In this case, we claim the second factor has no cohomology. 

Indeed, when $\lambda=(1^{p_1})$, we have that ${\bf  S}_{\lambda} = \bigwedge^{p_1}$ is an exterior power and $$\bigwedge^{p_1}\left(\mathcal{B}^\vee_2 \oplus \mathcal{B}^\vee_2\right) \cong \bigoplus_{j=0}^{p_1} \bigwedge^j \mathcal{B}^\vee_2 \otimes \bigwedge^{p_1-j} \mathcal{B}^\vee_2.$$ As in the above analysis of  \eqref{boomboom}, every summand $\bS_\mu(\mathcal{B}_2^\vee)$ which appears in the second factor then satisfies $1 \le \mu_1 \le t+1$ by the Pieri rule. Since $t+1\leq N \le \dim V_m-n= d-n$, we get vanishing by Borel-Weil-Bott and \eqref{conditiondual0}.

When $\lambda\neq (1^{p_1})$, letting $i$ denote the $(p_1, n)$-index of $\lambda$, the second factor
\[
  {\mathbf S}_{\lambda} (\mathcal B_2^{\vee}\oplus \mathcal B_2^{\vee})\otimes \bigwedge^{p_2} \mathcal B_2^{\vee} \otimes \cdots \otimes \bigwedge^{p_t} \mathcal B_2^{\vee}
\]
has no cohomology by Proposition~\ref{vanishdual}\,(3).
\end{proof}

\section{Cohomology vanishing on the Grassmannian}\label{s3}

\subsection{Overview} \label{ov} We establish the vanishing results which played a crucial role in the proofs of Propositions \ref{resolwedgep}, \ref{resolsymp} and \ref{resolup}. 

We continue to write $\mathbf G={\bf G}(d, n)$ for the Grassmannian of $n$-dimensional quotients of a $d$-dimensional vector space, 
and $\mathcal B\to \mathbf G$ for the tautological rank $n$ quotient. 

Recall from Section \ref{bwbs} that ${\mathbf S}_{\delta} \left(\mathcal B^{\vee}\right)$ has no cohomology provided that there exists $j$ such that 
\begin{equation}
\label{conditiondual}
j\leq \delta_j\leq d-n+j-1.
\end{equation} 

We will establish three results. The first corresponds to the resolution $\mathcal V_{\bullet}$, while the second pertains to the resolution $\mathcal W_{\bullet}$. Together, these  already capture the main ideas. The third result covers the resolutions $\mathcal X_{\bullet}, \mathcal Y_{\bullet}$ and $\mathcal U_{\bullet}$. Although the notation in this case is more involved, the argument does not require new ideas. We present these results separately for clarity. 

\begin{proposition}\label{vanishwedge}
  Let $\lambda\neq 0$ be a partition that fits in the $(2n) \times (d-n-1)$ rectangle and assume that $\lambda$ has $n$-index $i$. For every summand $\bS_\delta(\mathcal B^{\vee}) \subset \bS_\lambda(\mathcal B^{\vee} \oplus \mathcal B^{\vee}) \otimes \bigwedge^\bullet(\mathcal B)$, the partition $\delta$ satisfies condition \eqref{conditiondual} with $j=i$.

    In particular, all cohomology of $\bS_\lambda(\mathcal B^{\vee} \oplus \mathcal B^{\vee}) \otimes \bigwedge^\bullet(\mathcal B)$ vanishes.
        
  \end{proposition}
  
    \begin{proposition}   \label{vanishsym}  
  Let $\lambda\neq 0$ be a partition that fits in the $(2n) \times (d-n-1)$ rectangle and assume that $\lambda$ has $n$-index $i$.
  \begin{enumerate}[\rm (1)]
  \item If $i < n$, then for every summand $\bS_\delta(\mathcal B^{\vee}) \subset \bS_\lambda(\mathcal B^{\vee} \oplus \mathcal B^{\vee}) \otimes \Sym^\bullet(\mathcal B)$, the partition $\delta$ satisfies condition \eqref{conditiondual} with $j=i$.

    In particular, all cohomology of $\bS_\lambda(\mathcal B^{\vee} \oplus \mathcal B^{\vee}) \otimes \Sym^\bullet(\mathcal B)$ vanishes.    

  \item If $i = n$ and $k \le n$, then for every summand $\bS_\delta(\mathcal B^{\vee}) \subset \bS_\lambda(\mathcal B^{\vee} \oplus \mathcal B^{\vee}) \otimes \Sym^k(\mathcal B)$, the partition $\delta$ satisfies condition \eqref{conditiondual} with $j=i$.

    In particular, all cohomology of $\bS_\lambda(\mathcal B^{\vee} \oplus \mathcal B^{\vee}) \otimes \Sym^k(\mathcal B)$ vanishes.        
  \end{enumerate}
\end{proposition}

\begin{proposition}\label{vanishdual}
  Let $\lambda\neq 0$ be a partition that fits in the $(2n) \times (d-n-t-1)$ rectangle, for some $t\geq 0$. 
\begin{enumerate}
[\rm (1)]  
\item Assume that $\lambda$ has $n$-index $i$. Let $p_1, \ldots, p_t$ be nonnegative integers and let $0\leq k\leq n$. Then every summand $$\bS_{\delta}(\mathcal B^{\vee})\subset {\mathbf S}_{\lambda} (\mathcal B^{\vee}\oplus \mathcal B^{\vee})\otimes \bigwedge^{p_1} \mathcal B^{\vee} \otimes \cdots \otimes \bigwedge^{p_t} \mathcal B^{\vee}\otimes \bigwedge^k \mathcal B$$ satisfies condition \eqref{conditiondual} with $j=i$.
  
In particular, all cohomology of ${\mathbf S}_{\lambda} (\mathcal B^{\vee}\oplus \mathcal B^{\vee})\otimes \bigwedge^{p_1} \mathcal B^{\vee} \otimes \cdots \otimes \bigwedge^{p_t} \mathcal B^{\vee}\otimes \bigwedge^k \mathcal B$ vanishes. 

\item Assume that $\lambda$ has $n$-index $i$. Let $p_1, \ldots, p_t$ be nonnegative integers and let $0\leq k\leq n$. All summands of
  $${\bf S}_{\lambda}(\mathcal B^{\vee}\oplus \mathcal B^{\vee})\otimes \bigwedge^{p_1}\mathcal B^{\vee}\otimes \cdots \otimes \bigwedge^{p_t} \mathcal B^{\vee}\otimes \textnormal{Sym}^k (\mathcal B)$$
  satisfy condition \eqref{conditiondual} for $j=i$, and therefore this bundle has no cohomology. 

\item Assume $\lambda\neq (1^p)$ has $(p, n)$-index $i$, for some $0\leq p\leq n$. Let $p_1, \ldots, p_{t-1}$ be nonnegative integers. Then every summand
  \[
    \bS_{\delta}(\mathcal B^{\vee})\subset {\mathbf S}_{\lambda} (\mathcal B^{\vee}\oplus \mathcal B^{\vee})\otimes \bigwedge^{p_1} \mathcal B^{\vee} \otimes \cdots \otimes \bigwedge^{p_{t-1}} \mathcal B^{\vee}
  \]
  satisfies condition \eqref{conditiondual} with $j=i$. 

  In particular, all cohomology of ${\mathbf S}_{\lambda} (\mathcal B^{\vee}\oplus \mathcal B^{\vee})\otimes \bigwedge^{p_1} \mathcal B^{\vee} \otimes \cdots \otimes \bigwedge^{p_{t-1}} \mathcal B^{\vee}$ vanishes. 
  \end{enumerate}
\end{proposition}

\subsection{Littlewood--Richardson coefficients} \label{ss:L-R}

For the proofs of the above propositions, we need a few preliminaries about the Littlewood--Richardson coefficients. The material below is well-known, but to establish the notation, we recall several definitions and basic facts, some of which can be found in \cite[\S\S 2, 3]{SS}. 

For two partitions $\alpha$ and $\beta$ of the same size $|\alpha|=|\beta|$, write $\alpha \ge \beta$ ($\alpha$ {\bf dominates} $\beta$) if, for all $m$, we have $$\alpha_1+\cdots+\alpha_m \ge \beta_1+\cdots+\beta_m.$$

Note that $\alpha \ge \beta$ if and only if $\alpha^\dagger \le \beta^\dagger$.

Given partitions $\alpha, \beta, \gamma$, let $c_{\alpha, \beta}^\gamma$ denote the Littlewood--Richardson coefficient, which is the multiplicity of the Schur functor $\bS_\gamma$ in the tensor product $\bS_\alpha \otimes \bS_\beta$. 

The coefficient $c_{\alpha, \beta}^{\gamma}$ counts the number of Littlewood--Richardson tableaux. These are fillings of the skew tableau of shape $\gamma/\alpha$ with content $\beta$ (i.e., for all $i$, the label $i$ appears exactly $\beta_i$ times) such that the following two properties hold:
\begin{itemize}
\item ({\it Semistandard}) In each row, the entries are weakly increasing from left to right, and in each column, the entries are strictly increasing from top to bottom.
\item ({\it Lattice word property}) Let $w$ be the word (called reading word) obtained by reading the entries in each row from right to left, starting with the top row and going down. For each $i$ and $m$, let $w_i(m)$ be the number of times that $i$ appears in the first $m$ entries of $w$. Then for all $m$ and $i$, we have $w_i(m) \ge w_{i+1}(m)$.
\end{itemize}
  
We collect a few facts about these coefficients in the next result.

\begin{proposition} \label{prop:LR}
\begin{enumerate}[\rm (1)]
\item For any complex vector bundles $V,$ $W,$ we have
  \[
    \bS_\gamma(V \oplus W) \cong \bigoplus_{\alpha, \beta} (\bS_\alpha (V) \otimes \bS_\beta(W))^{\oplus c^\gamma_{\alpha, \beta}}
  \]
  where the sum is over all partitions $\alpha, \beta$.

\item If $c^\gamma_{\alpha, \beta} \ne 0$, then $|\gamma| = |\alpha| + |\beta|$.

\item If $c^\gamma_{\alpha, \beta} \ne 0$, then $\gamma$ contains both $\alpha$ and $\beta$, i.e., $\gamma_i \ge \max(\alpha_i, \beta_i)$ for all $i$.
  
\item In a Littlewood--Richardson tableau of shape $\gamma / \alpha$ and type $\beta$, all occurrences of the number $i$ must appear in rows $i$ and later.

  As a consequence, if $c^\gamma_{\alpha, \beta} \ne 0$, then $\alpha+\beta$ dominates $\gamma$, i.e., for all $m$, we have
  \[
    \sum_{i=1}^m (\alpha_i + \beta_i) \ge \sum_{i=1}^m \gamma_i.
  \]
  
\item If $c^\gamma_{\alpha, \beta} \ne 0$, then $\gamma$ dominates $\alpha \cup \beta$ (this is the partition obtained from all of the rows of $\alpha$ and $\beta$ placed one after the other according to their lengths).
\end{enumerate}
\end{proposition}

\begin{proof}
  (1) see \cite[(4.5)]{SS} for a derivation.

  (2) and (3) are clear from the interpretation in terms of tableaux.

  (4) We prove this by induction on $i$. If $i=1$, there is nothing to show. Now suppose the statement is true for $i$. Suppose that there is a Littlewood--Richardson tableau in which $i+1$ appears in the first $i$ rows. Let $w$ be the reading word of this tableau. By the lattice word property, this instance of $i+1$ cannot appear in a row before the earliest (relative to $w$) instance of $i$, so it must appear in row $i$, and it must appear to the left of $i$ in the tableau. However, this violates the semistandard condition.

  (5) This is a consequence of (4) since $\alpha \cup \beta = (\alpha^\dagger + \beta^\dagger)^\dagger$ and $c^\gamma_{\alpha, \beta} = c^{\gamma^\dagger}_{\alpha^\dagger, \beta^\dagger}$.
\end{proof}

\subsection{Lemmas} To carry out the proofs of Propositions \ref{vanishwedge}, \ref{vanishsym}, and \ref{vanishdual}, we first establish a few supporting results. 

First, by Proposition \ref{prop:LR}\,(1), we have
\begin{align} \label{eqn:decomp}
  \bS_\lambda(\mathcal B^{\vee} \oplus \mathcal B^{\vee}) \cong \bigoplus_{\alpha, \beta} (\bS_\alpha (\mathcal B^{\vee}) \otimes \bS_\beta (\mathcal B^{\vee}))^{\oplus c^{\lambda}_{\alpha, \beta}} \cong \bigoplus_{\alpha, \beta, \gamma} \bS_\gamma(\mathcal B^{\vee})^{\oplus c^\lambda_{\alpha, \beta} c^\gamma_{\alpha, \beta}}.
\end{align} Here, the number of rows of the partitions $\alpha, \beta, \gamma$ is less than or equal to $n$, while the number of rows in the partition $\lambda$ is less than or equal to $2n$. 

We assume first that $\lambda$ is contained in the $(2n) \times (d-n-1)$ rectangle as needed in Proposition \ref{vanishwedge} and \ref{vanishsym}.

{\bf Reserve $i$ to be the $n$-index of $\lambda$. }

Pick a triple $\alpha, \beta, \gamma$ such that $\bS_\gamma(\mathcal B^{\vee}) \ne 0$, and $c_{\alpha, \beta}^\lambda c_{\alpha,\beta}^\gamma \ne 0$. We will deduce a number of restrictions on the partitions $\alpha, \beta, \gamma$. 

\begin{lemma} \label{lem:alphai-ge-i}
  We have $\alpha_i \ge i$.
\end{lemma}

\begin{proof}
  Suppose that $\alpha_i < i$. Since $\lambda$ has $n$-index $i$, recall from \eqref{indin} that $\lambda_{i+1}=\cdots=\lambda_{i+n}=i$ and thus $\lambda_i\geq i$. Then the $i$th column of the skew shape $\lambda / \alpha$ has at least $n+1$ boxes (in rows $i$ through $n+i$). But then any valid Littlewood--Richardson tableau of shape $\lambda / \alpha$ needs at least $n+1$ labels (because of the semistandard condition). This implies that $\beta_{n+1} > 0$, contradicting the fact that $\beta$ has at most $n$ rows. 
 \end{proof}

\begin{lemma} \label{lem:alpha+beta}
  We have
  \[
    (\alpha_1+\beta_1) + \cdots + (\alpha_i + \beta_i) \le i(d-n+i-1).
  \]
\end{lemma}

\begin{proof}
\begin{subeqns}
  By Proposition~\ref{prop:LR}\,(5), we know that $\lambda \ge \alpha \cup \beta$, so that
  \begin{equation}\label{ineq}
    \begin{split}
    \lambda_1 + \cdots + \lambda_{2i} &\ge (\alpha \cup \beta)_1 + \cdots + (\alpha \cup \beta)_{2i}\\
    &\ge (\alpha_1 + \beta_1) + \cdots + (\alpha_i + \beta_i).
  \end{split}
\end{equation}
  Since $\lambda_j \le d-n-1$ for $j=1,\dots,i$ and $\lambda_j \le i$ for $j=i+1,\dots,2i$, the lemma follows. 
\end{subeqns}
\end{proof}

\begin{lemma} \label{lem:gammai-bounds}
We have  $i + 1 \le \gamma_i \le d-n+i-1$.
\end{lemma}

\begin{proof}
  \begin{subeqns}
    We know that $\alpha_i \ge i$ from Lemma~\ref{lem:alphai-ge-i}. If in fact $\alpha_i \ge i+1$, then we can use Proposition~\ref{prop:LR}(3) to conclude that $\gamma_i \ge i+1$.

    Otherwise, we have $\alpha_i = i$. Suppose that $\gamma_i = i$. Then the $i$th row of $\gamma/\alpha$ has no boxes. Since $c^\lambda_{\alpha, \beta} \ne 0$, there is a Littlewood--Richardson tableau of shape $\lambda/\alpha$ and type $\beta$. Since $\lambda_{i+n}=i$, the $i$th column of $\lambda/\alpha$ has at least $i+n-\alpha_i^\dagger$ boxes, so that $\beta$ has at least $i+n-\alpha_i^\dagger$ rows (by the semistandard condition).

    Next, there is also a Littlewood--Richardson tableau of shape $\gamma/\alpha$ and type $\beta$. The integers in the interval $[i,i+n-\alpha_i^\dagger]$ cannot go in the first $i-1$ rows of $\gamma/\alpha$ by Proposition~\ref{prop:LR}(4), and cannot go in the $i$th row since it is empty. Thus, these numbers must go in rows $i+1$ or higher.  
  Again, since $\gamma_i=i$, they are also constrained to the first $i$ columns of $\gamma/\alpha$ as well. 
Now, in $\gamma/\alpha$, the $i$th column only has boxes in rows $\alpha_i^\dagger+1, \ldots, n$, at most. Consequently, the labels $[i, i+n-\alpha_i^{\dagger}]$ can only be placed in rows $\alpha_i^\dagger+1,\dots,n$. 

Suppose it is possible to do this. Consider the subdiagram $D$ of $\gamma/\alpha$ consisting of boxes that are filled with entries $\ge i$. If we subtract $i-1$ from every entry, we claim that the result is a valid Littlewood--Richardson tableau of shape $D$. Subtracting the same amount from each entry does not affect any of the semistandard inequalities. Furthermore, if $w$ is the reading word of the Littlewood--Richardson tableau of $\gamma/\alpha$ of type $\beta$ that we're considering, and $w'$ is the reading word of its restriction to $D$, then in the notation of \S\ref{ss:L-R}, for $j \ge i$ and any $m$, we have $w'_j(m) = w_j(m)$. In particular, then $w'_j(m) \ge w'_{j+1}(m)$ for all $j \ge i$ and hence the result of subtracting $i-1$ from all entries of $D$ has the lattice word property.

But then we have too many labels: in fact, at least $n-\alpha_i^\dagger+1$ labels and only $n-\alpha_i^\dagger$ rows to put them into. We have a contradiction to Proposition~\ref{prop:LR}(4) and hence $\gamma_i \ge i+1$.

Finally, again by Proposition~\ref{prop:LR}\,(4), we have $\alpha+\beta \ge \gamma$. Using Lemma~\ref{lem:alpha+beta}, we obtain 
\begin{equation}\label{ineq2}
  i\gamma_i \le \gamma_1 + \cdots + \gamma_i \le i(d-n+i-1),
\end{equation}
and hence $\gamma_i \le d-n+i-1$.\footnote{We could relax the condition that $\lambda$ is contained in the $(2n) \times (d-n-1)$ rectangle here. It would suffice to know that $\lambda_1+\cdots+\lambda_i \le i(d-n-1)+i-1$.}
\end{subeqns}
\end{proof}

\begin{lemma} \label{lem:sym-rect}
  \begin{enumerate}
  \item   If $i < n$, then $\gamma_{i+1} \ge i$.
  \item   If $i=n$, then $\gamma_n \ge 2n$.
  \end{enumerate}
\end{lemma}

\begin{proof} Since $\lambda_{i+1}=i$, we have $\alpha_{i+1}\leq i$ by Proposition \ref{prop:LR}\,(3). 
 
 Let $c = i - \alpha_{i+1}$. Then $\lambda/\alpha$ contains the subrectangle occupying rows $i+1,\dots,i+n$ and columns $\alpha_{i+1}+1,\dots,i$, which in particular has $n$ rows and $c$ columns. Since $c^\lambda_{\alpha, \beta} \ne 0$, we can fill $\lambda/\alpha$ with content $\beta$. By examining the $n\times c$ subrectangle and using the semistandard property, we obtain $\beta_n \ge c$. Let $\beta'$ be the result of subtracting $c$ from all parts of $\beta$. Then $$\bS_\beta(\mathcal B^{\vee}) = (\det \mathcal B^{\vee})^{\otimes c} \otimes \bS_{\beta'}(\mathcal B^{\vee})$$ since $\rank(\mathcal B^{\vee})=n$. Hence to compute $\bS_\alpha(\mathcal B^{\vee}) \otimes \bS_\beta(\mathcal B^{\vee})$, we can first add $c$ to all values of $(\alpha_1,\dots,\alpha_n)$ and then tensor with $\bS_{\beta'}(\mathcal B^{\vee})$.

  In particular, if $i<n$, then $\gamma_{i+1} \ge \alpha_{i+1} + c = i$, again by Proposition \ref{prop:LR}(3). Otherwise, if $i=n$, since $\alpha_{n+1}=0$ we find $c=n$. Using Lemma~\ref{lem:alphai-ge-i}, we have $\alpha_n\geq n$, and thus $\gamma_n \ge \alpha_n + c \ge 2n$.
\end{proof}

\subsection{Vanishing} We continue to use the notation from the previous section.\vskip.1in

\begin{proof}[Proof of Proposition \ref{vanishwedge}]
Assume $\lambda$ fits in the $(2n)\times (d-n-1)$ rectangle, and let ${\bf S}_{\gamma}(\mathcal B^{\vee})$ be a summand of ${\bf S}_{\lambda}(\mathcal B^{\vee}\oplus \mathcal B^{\vee}).$  Consider the tensor product $\bS_\gamma (\mathcal B^{\vee}) \otimes \bigwedge^k(\mathcal B)$. First we use that $\bigwedge^k\mathcal B = \det(\mathcal B) \otimes \bigwedge^{n-k} \mathcal B^{\vee}$ and $\bigwedge^{n-k} \mathcal B^{\vee} = \bS_{(1^{n-k})} (\mathcal B^{\vee})$. The Pieri rule describes the outcome of tensoring with $\bigwedge^{n-k} \mathcal B^{\vee}$. The result is a sum over partitions where we add $n-k$ boxes, no two in the same row. Tensoring with $\det(\mathcal B)$ is the same as subtracting 1 from all entries. Therefore, for any summand $\bS_\delta (\mathcal B^{\vee})$ of this tensor product, we have $\gamma_i-1 \le \delta_i \le \gamma_i$. Hence we conclude from Lemma~\ref{lem:gammai-bounds} that
$$i \le \delta_i \le d-n+i-1,$$ completing the argument in this case. 
\end{proof}

\begin{proof}[Proof of Proposition \ref{vanishsym}]
The Pieri rule applied to symmetric powers tells us that $\bS_\nu(\mathcal B^{\vee})$ is a summand of $\bS_\mu(\mathcal B^{\vee}) \otimes \Sym^k(\mathcal B^{\vee})$ if and only if $|\nu|=|\mu|+k$ and the interlacing property $\nu_j \ge \mu_j \ge \nu_{j+1}$ holds for all $j$. In fact, it makes no difference if some entries of $\nu$ and $\mu$ are negative since we can make them nonnegative by twisting by powers of $\det(\mathcal B^{\vee})$ and untwisting after. 

If $\bS_\delta(\mathcal B^{\vee})$ is a summand of $\bS_\gamma(\mathcal B^{\vee}) \otimes \Sym^k(\mathcal B)$, we obtain by dualizing that $\bS_{-\delta}(\mathcal B^{\vee})$ is a summand of 
$\bS_{-\gamma}(\mathcal B^{\vee})\otimes \Sym^k (\mathcal B^{\vee})$. Thus, $|\gamma|=|\delta|+k$ and the interlacing property gives $$\gamma_{j+1}\leq \delta_j\leq \gamma_j.$$ (In particular, if $\bS_\delta(\mathcal B^{\vee})$ is a summand of $\bS_\gamma(\mathcal B^{\vee}) \otimes \Sym^k(\mathcal B)$, then $\bS_\gamma(\mathcal B^{\vee})$ is a summand of $\bS_\delta(\mathcal B^{\vee}) \otimes \Sym^k(\mathcal B^{\vee})$.)

Consider the $n$-index $i$ of $\lambda$. If $i<n$, then Lemma~\ref{lem:sym-rect} tells us that $\gamma_{i+1} \ge i$. The interlacing property then forces $\delta_i \ge i$. If $i=n$, then $\gamma_n \ge 2n$. Since $\gamma$ is obtained from $\delta$ by adding $k$ boxes, we have $$\delta_n \ge \gamma_n-k \ge 2n-k\geq n$$ since we assume that $k \le n$.

In any case, under either assumption, we have shown that $\delta_i \ge i$ and also $\delta_i \le \gamma_i \le d-n+i-1$ by Lemma~\ref{lem:gammai-bounds}. This is what we set out to prove. 
\end{proof}

\begin{proof}[Proof of Proposition \ref{vanishdual}]
  Assume now that $\lambda$ is contained in the $(2n) \times (d-n -t-1)$ rectangle. For case (1), for a partition $\lambda$ of $n$-index $i$, we have $$\lambda_j\leq d-n-t-1 \text{ for }j\leq i, \quad \lambda_{j}\leq i \text{ for } i+1\leq j\leq 2i.$$ Thus, by \eqref{ineq} and \eqref{ineq2}, we have
  \[
    i\gamma_i\leq \gamma_1+\cdots+\gamma_i\leq \lambda_1+\cdots+\lambda_{2i}\leq i(d-n-t-1)+i\cdot i\implies \gamma_i\leq d-n-t-1+i.
  \]
  We also have $\gamma_i\geq i+1$ by Lemma \ref{lem:gammai-bounds}. We inspect the summands $${\mathbf S}_{\delta}(\mathcal B^{\vee})\subset {\mathbf S}_{\gamma}(\mathcal B^{\vee})\otimes \bigwedge^{p_1} \mathcal B^{\vee}\otimes \bigwedge^{p_2} \mathcal B^{\vee} \otimes \cdots \otimes \bigwedge^{p_t} \mathcal B^{\vee}\otimes \bigwedge^k \mathcal B.$$ We use $\bigwedge^k \mathcal B=\bigwedge^{n-k} \mathcal B^{\vee}\otimes \det \mathcal B.$ By repeated application of the Pieri rule, and taking into account that tensorization by $\det \mathcal B$ subtracts one box from each entry, we conclude $$\gamma_i-1\leq \delta_i\leq \gamma_i+t.$$ The conclusion follows since $$i+1\leq \gamma_i\leq d-n-t-1+i\implies i\leq \delta_i\leq d-n+i-1.$$ 

For part (2), we established in the proof of Proposition \ref{vanishsym} that all summands ${\bf S}_{\delta}(\mathcal B^{\vee})$ of ${\bf S}_{\lambda}(\mathcal B^{\vee}\oplus \mathcal B^{\vee})\otimes  \text{Sym}^k \,\mathcal B$ satisfy $$i\leq \delta_i\leq d-n-t-1+i,$$ where the modified upper bound is due to the different size of the rectangle that contains $\lambda$. By the Pieri rule, tensorization by $(\bigwedge^{\bullet} \mathcal B^{\vee})^{\otimes t}$ can only increase the lengths of rows, and if so by at most $t$ boxes. Thus all summands ${\bf S}_{\nu}(\mathcal B^{\vee})$ of ${\bf S}_{\lambda}(\mathcal B^{\vee}\oplus \mathcal B^{\vee})\otimes \bigwedge^{p_1}\mathcal B^{\vee}\otimes \cdots \otimes \bigwedge^{p_t} \mathcal B^{\vee}\otimes \text{Sym}^k \,\mathcal B$ satisfy $$i\leq \nu_i\leq d-n+i-1,$$ as claimed.

Finally, we prove (3).  If $\mathbf S_{\gamma}(\mathcal B^{\vee})$ is a summand of ${\mathbf S}_{\lambda} (\mathcal B^{\vee}\oplus \mathcal B^{\vee})$, then by the same reasoning as in Lemma \ref{lem:alphai-ge-i} we have $i\leq \alpha_i\leq \gamma_i$. By \eqref{ineqb} $$\lambda_1, \ldots, \lambda_i\leq d-n-t-1, \quad \lambda_{i+1}, \ldots, \lambda_{2i}\leq i+1.$$ 
Using \eqref{ineq} and \eqref{ineq2}, we obtain
\[
  i\gamma_i\leq \gamma_1+\cdots+\gamma_i\leq \lambda_1+\cdots+\lambda_{2i}\leq i(d-n-t-1)+i(i+1)\implies \gamma_i\leq d-n-t+i.
\]
By repeated application of the Pieri rule, all summands $\mathbf S_{\delta}\left(\mathcal B^{\vee}\right)$ of $\mathbf S_{\gamma}(\mathcal B^{\vee})\otimes 
\bigwedge^{p_1} \mathcal B^{\vee} \otimes \cdots \otimes \bigwedge^{p_{t-1}} \mathcal B^{\vee}$ satisfy $\gamma_i\leq \delta_i\leq \gamma_i+(t-1).$ Since $i\leq \gamma_i\leq d-n-t+i$, we have $i\leq \delta_i\leq d-n+i-1.$ Therefore, condition \eqref{conditiondual} is satisfied for $\delta$ and $j=i$, and all cohomology vanishes.
\end{proof}

\section{Corollaries}\label{cor} 

\subsection{Corollary \ref{tos} and universality.} \label{cortos}

We explain the universality arguments needed to derive Corollary \ref{tos} from the genus $0$ computations in Theorems \ref{ta}, \ref{tb}, and \ref{tc}. 

Regarding equation \eqref{wedge}, we have the factorization
\begin{equation}\label{factor}
  \sum_{n=0}^{\infty} q^n\chi \bigg(\quot_{\,C }(\mathbb C^N, n), \bigwedge_yL^{[n]}\bigg)=\mathsf A^{\chi(\mathcal O_C)}\cdot \mathsf B^{\chi(L)}
\end{equation}
where $\mathsf A, \mathsf B\in 1+q\,\mathbb Q[y][\![q]\!]$ are two universal power series whose coefficients may depend on $N$ but not on the pair $(C, L)$. This factorization is by now a standard fact, see for instance \cite {EGL, OS, stark} for various incarnations of this statement. To establish \eqref{wedge}, we show $$\mathsf A=(1-q)^{-1}, \quad \mathsf B=(1+qy)^N.$$ 

Specializing $C=\mathbb P^1$ and $\deg L=\ell\geq n$ in \eqref{factor}, and using Theorem \ref{ta}\,(1) we obtain
\[
  \left[q^n\right] \mathsf A\cdot \mathsf B^{\ell+1}=\sum_{k=0}^{n} y^k \binom {N\chi (L)}{k},
\]
where the brackets denote extracting the relevant coefficient in the $q$-expansion. By direct calculation, we also have
$$
\left[q^n\right] (1-q)^{-1}\cdot \left((1+qy)^{N}\right)^{\ell+1}=\sum_{k=0}^{n} y^k \binom {N\chi (L)}{k}.
$$
It remains to explain that the coefficients
$$
\left[q^n\right] \mathsf A\cdot \mathsf B^{\ell+1} \text{ for all } \ell\geq n
$$
determine the series $\mathsf A, \mathsf B$ at most uniquely. We argue inductively, each coefficient at a time. Explicitly, we write
$$\mathsf A=1+a_1q+a_2q^2+\cdots, \quad \mathsf B=1+b_1q+b_2q^2+\cdots.
$$
Then $$\left[q^n\right] \mathsf A\cdot \mathsf B^{\ell+1}=a_n+(\ell+1)b_n + \text {lower order terms in }n.$$ The lower order terms are determined by the induction hypothesis. The inductive step follows since the principal terms $a_n+(\ell+1)b_n$ for all $\ell \geq n$ determine $a_n, b_n$ at most uniquely. 

For \eqref{dual} the argument is similar, using the factorization
$$\sum_{n=0}^{\infty} q^n\chi\bigg(\quot_{\,C }(\mathbb C^N, n), \bigotimes_{i=1}^{t}\big(\bigwedge_{y_i} M_i^{[n]} \big)^{\vee}\bigg)=
{\mathsf A}^{\chi(\mathcal O_C)} \cdot \mathsf B_1^{\chi(M_1)} \cdots \mathsf B_{t} ^{\chi(M_t)}.$$
This time, we specialize $C=\mathbb P^1$, and $$M_1= M_2=\cdots=M_t=M, \quad \deg M=m\geq n.$$ By Theorem \ref{tc}, we have 
$$\left[q^n\right]\mathsf A\cdot (\mathsf B_1\cdots \mathsf B_t)^{m+1}=1 \text{ for all } m\geq n.$$ Indeed, the only non-zero contribution appears from the free term $y_1=\cdots=y_t=0$ and yields the answer $\chi(\mathsf {Quot}_{\mathbb P^1}(\mathbb C^N, n), \mathcal O)=1$ since $\mathsf{Quot}_{\mathbb P^1}(\mathbb C^N, n)$ is rational. By the above reasoning, the series
$$\mathsf A=(1-q)^{-1}, \quad \mathsf B_1\cdots \mathsf B_t=1$$
are uniquely determined. Next, we set $$M_1=L, \quad M_2=\cdots=M_t=M,$$
where $\deg L=\deg M-1=m-1\geq n-1$. This time around, Theorem~\ref{tc} implies
$$
\left[q^n\right] \left(\mathsf A\cdot \mathsf B_1^{-1}\right)\cdot \left(\mathsf B_1\cdots \mathsf B_t\right)^{m+1}=1 \implies \mathsf A\cdot \mathsf B_1^{-1}=(1-q)^{-1}, \quad  \mathsf B_1\cdots \mathsf B_t=1.
$$
Therefore $$\mathsf A=(1-q)^{-1}, \quad \mathsf B_1=\cdots=\mathsf B_t=1,$$ and \eqref{dual} follows from here. 

Equation \eqref{sym} uses Theorem \ref{tb}\,(1). Indeed, we have the factorization $$\sum_{n=0}^{\infty} q^n\chi \bigg(\quot_{\,C }(\mathbb C^N, n), \text{Sym}_yL^{[n]}\bigg)=\mathsf A^{\chi(\mathcal O_C)}\cdot \mathsf B^{\chi(L)}\,,$$ where $\mathsf{A}, \mathsf B\in 1+q\mathbb Q[\![y]\!][\![q]\!].$ Fix $n\geq 1$. By Theorem \ref{tb}, if $\deg L=\ell\geq n$, we have \begin{eqnarray*}\left[q^n\right]\,\mathsf A\cdot \mathsf B^{\ell+1} &= \sum_{k\geq 0} y^k \chi \bigg(\quot_{\,\mathbb P^1 }(\mathbb C^N, n), \text{Sym}^k L^{[n]}\bigg)= \sum_{k=0}^{n} y^k (-1)^k\binom {-N(\ell+1)}{k} \mod y^{n+1}.\end{eqnarray*} Both sides of this identity are polynomials in $(\ell+1)$. (On the left hand side, these polynomials depend on the first $q$-coefficients $a_1, \ldots, a_n, b_1, \ldots, b_n$ of $\mathsf A, \mathsf B$ considered modulo $y^{n+1}$.) Consequently, the same equality holds for all values of $\ell$ without restrictions. Hence, for all $L$, we have 
\begin{align*}\sum_{n=0}^{\infty}\sum_{k=0}^{n} q^n y^k \chi \bigg(\quot_{\,\mathbb P^1 }(\mathbb C^N, n), \text{Sym}^k L^{[n]}\bigg)&= \sum_{n=0}^{\infty}q^n \bigg(\sum_{k=0}^{n} y^k (-1)^k\binom {-N(\ell+1)}{k} \bigg)\\&=(1-q)^{-1}\cdot (1-qy)^{-N\chi(L)}\end{align*} as stated in \eqref{sym}.  \qed

\subsection {Corollary \ref{t2}.} \label{prooft2} We analyze the cohomology groups of $L^{[n]}$ for all line bundles $L\to \mathbb P^1$ using a few simple considerations. The corollary can also be derived by combining the methods of \cite [Corollary 9.3]{BGS} when adapted to the case of the projective line, followed by a calculation on the symmetric product.
 
Let $p\in \mathbb P^1$ and write $\mathcal Q_p=\mathcal Q|_{p\times \mathsf{Quot}}.$ The exact sequence $$0\to L(-p)\to L\to L_{p}\to 0$$ yields an exact sequence over $\mathsf {Quot}$: 
\begin{equation}\label{ex}
  0\to L(-p)^{[n]}\to L^{[n]}\to \mathcal Q_p\to 0.
\end{equation}

When $\deg L\geq n+1$, the bundles $L^{[n]}$ and $L(-p)^{[n]}$ carry no higher cohomology by Theorems \ref{ta}\,(1) or \ref{tb}\,(1) for $k=1$. Taking cohomology in \eqref{ex}, we obtain
\begin{equation}\label{vann}
  H^{i}(\mathcal Q_p)=0, \quad i\geq 1.
\end{equation}
We go back to \eqref{ex} written for arbitrary $L$, not necessarily sufficiently positive. Considering cohomology again and using \eqref{vann}, we obtain
\begin{equation}\label{van}
  H^i(L^{[n]})=H^i(L(-p)^{[n]}), \quad i\geq 2.
\end{equation}
Since $H^i(L^{[n]})=0$ for $\deg L\geq n$ and $i\geq 2$, it follows from \eqref{van} that $H^i(L^{[n]})=0$ for all $i\geq 2$ and all $L$. \qed

 \end{document}